\documentclass[11pt]{amsart}

\usepackage[dvipsnames]{xcolor}
\usepackage{amssymb,verbatim}
\usepackage{amsmath,amsfonts,enumitem,bm}
\usepackage[mathscr]{euscript} %
\usepackage{amsthm}
\usepackage{stmaryrd}
\usepackage{url}
\usepackage{graphicx} %
\usepackage{enumitem} %
\usepackage{hyperref} %
\hypersetup{colorlinks} %
\usepackage{bbm} %
\usepackage{bm} %
\usepackage{mathrsfs} %
\usepackage{cancel} %
\usepackage{ytableau} %
\usepackage{mathtools}
\usepackage{faktor} %
\usepackage{algorithm}
\usepackage{algorithmic}
\usepackage{bm}
\usepackage{subfigure}
\usepackage{graphicx}

\usepackage[colorinlistoftodos]{todonotes}

\RequirePackage{cleveref}
\usepackage{hypcap}
\hypersetup{colorlinks=true, citecolor=darkblue, linkcolor=darkblue}
\definecolor{darkblue}{rgb}{0.0,0,0.7}
\newcommand{\darkblue}{\color{darkblue}}

\definecolor{darkred}{rgb}{0.68,0,0}
\newcommand{\darkred}{\color{darkred}}

\definecolor{darkgreen}{rgb}{0,.38,0}
\newcommand{\darkgreen}{\color{darkgreen}}

\newcommand{\defn}[1]{\emph{\darkblue #1}}
\newcommand{\defna}[1]{\emph{\darkred #1}}

\newcommand{\defng}[1]{\emph{\darkgreen #1}}

\setlist[enumerate]{
	label=\textnormal{({\roman*})},
	ref={\roman*}}

\makeatletter
\def\th@plain{%
	\thm@notefont{}%
	\itshape %
}
\def\th@definition{%
	\thm@notefont{}%
	\normalfont %
}
\makeatother

\newtheorem{thm}{Theorem}[section]
\newtheorem*{thm*}{Theorem}
\newtheorem{lemma}[thm]{Lemma}

\newtheorem*{claim*}{Claim}

\newtheorem{prop}[thm]{Proposition}
\newtheorem{conj}[thm]{Conjecture}
\newtheorem{conv}[thm]{Convention}

\theoremstyle{definition}
\newtheorem{ex}[thm]{Example}
\newtheorem{rem}[thm]{Remark}

\numberwithin{figure}{section}
\numberwithin{equation}{section}

\def\wh{\widehat}

\def\emp{\nothing}

\def\zz{\mathbb Z}
\def\nn{\mathbb N}
\def\cc{\mathbb C}

\def\rr{\mathbb R}

\def\kk{\mathbb K}

\def\bi{{\bm{i}}}

\def\la{\lambda}
\def\ga{\gamma}
\def\si{\sigma}

\def\al{\alpha}
\def\be{\beta}
\def\om{\omega}

\def\vk{\varkappa}

\def\CS{\Phi}

\def\cL{\mathcal L}

\def\cw{\mathcal \om}

\def\ssu{\subset}

\def\wt{\widetilde}
\def\<{\langle}
\def\>{\rangle}

\def\Bot{\text{{\rm bot}}}
\def\Im{\text{{\rm Im}}}

\def\y{ {\text {\rm y}  } }

\def\Z{\zz}

\def\GL{ {\text {\rm GL} } }

\def\0{{\mathbf 0}}
\def\id{{\rm id}}

\def\nothing{\varnothing}

\def\.{\hskip.06cm}
\def\ts{\hskip.03cm}

\def\dist{{\text {\rm dist} }}

\def\bx{{\textbf{\textit{x}}}}

\def\by{{\textbf{\textit{y}}}}

\def\bal{{\boldsymbol{\alpha}}}

\newcommand{\sign}{\mathrm{sign}}

\newcommand{\lcm}{\mathrm{lcm}}

\newcommand{\Range}{\mathrm{Range}}
\newcommand{\Exc}{\mathrm{Exc}}

\def\.{\hskip.06cm}
\def\ts{\hskip.03cm}

\def\nin{\noindent}

\def\Nu{\wh u}
\def\Nv{\wh v}
\def\Nw{\wh w}

\newcommand{\textsu}[1]{\textup{\textsf{#1}}}
\usepackage{indentfirst}
\DeclareTextSymbolDefault{\ae}{T1}

\newcommand{\ComCla}[1]{\textup{\textsu{#1}}}

\newcommand{\sharpP}{\ComCla{\#P}}
\newcommand{\SP}{\ComCla{\#P}}

\newcommand{\GapP}{\ComCla{GapP}}

\newcommand{\BPP}{\ComCla{BPP}}

\def\SP{\sharpP}

\def\revdom{\text{\rm revdom}}

\def\xx{\textbf{\textit{x}}}

\newcommand{\inv}{\operatorname{{\ell}}}
\newcommand{\Des}{\operatorname{{\rm Des}}}

\newcommand{\Sch}{\mathfrak{S}} %

\newcommand{\Sc}{\Sch}

\newcommand{\code}{{\sf{code}}}
\newcommand\wtx{{\tt wt}}
\newcommand\pd{{\tt PD}}

\newcommand{\satA}[1]{*_{#1}}

\makeatletter
\newcommand{\gettikzxy}[3]{%
  \tikz@scan@one@point\pgfutil@firstofone#1\relax
\pgfmathsetmacro{#2}{\the\pgf@x/\linkpatternunit}
\pgfmathsetmacro{#3}{\the\pgf@y/\linkpatternunit}
}
\tikzset{label anchor/.code={%
    \let\tikz@auto@anchor=\pgfutil@empty
    \def\tikz@anchor{#1}
  },
  label anchor/.default=center
}
\makeatother
\tikzset{arrow/.style={postaction={decorate,thick,decoration={markings,mark = at position #1 with {\arrow{>}}}}},arrow/.default=0.5}
\tikzset{invarrow/.style={postaction={decorate,thick,decoration={markings,mark = at position #1 with {\arrow{<}}}}},invarrow/.default=0.5}
\newdimen\linkpatternunit%
\newcount\linkpatternsize%
\newcount\lpsize%
\newif\iflinkpatterninverted%
\newif\iflinkpatterntikzstarted%
\newif\iflinkpatternboxed%
\newif\iflinkpatternaxis%
\newif\iflinkpatternstraightlines%
\newif\iflinkpatternnumbered%
\newif\iflinkpatternalias%
\newif\iflinkpatternnode%
\newif\iflinkpatterncentered%
\newcount\linkpatternfused%
\pgfkeys{/linkpattern/.cd,centered/.is if=linkpatterncentered,inverted/.is if=linkpatterninverted,numbered/.is if=linkpatternnumbered,tikzstarted/.is if=linkpatterntikzstarted,straight lines/.is if=linkpatternstraightlines,boxed/.is if=linkpatternboxed,axis/.is if=linkpatternaxis,vertexcolor/.store in=\linkpatternvertexcolor,
edgecolor/.store in=\linkpatternedgecolor,
boxcolor/.code={\linkpatternboxedtrue\def\linkpatternboxcolor{#1}},tikzoptions/.style={every linkpattern/.append style={#1}},
size/.code={\linkpatternsize=#1},
numbering0/.code={\def\lpnumbering{#1}\def\linkpatternnumbering{#1}},
numbering/.style={numbered,numbering0={#1}},
unit/.code={\linkpatternunit=#1},
height/.store in=\linkpatternheight,
shape/.store in=\linkpatternshape,
looseness/.store in=\linkpatternlooseness,
squareness/.store in=\linkpatternsquareness,
extra space/.store in=\linkpatternextraspace,
width/.store in=\linkpatternwidth,
alias/.is if=linkpatternalias,
pos/.store in=\linkpatternpos,
labeloptions/.style={labeloptionslist/.append style={#1}},
labeloptionslist/.style={inner sep=2pt,font=\scriptsize,
execute at begin node=$,execute at end node=$,
label anchor={#1+180}},
nodeon/.is if=linkpatternnode,
node/.style={nodeon,nodeoptionslist/.append style={#1}},
nodeoptionslist/.style={},
pipedream/.style={shape=pipedream,looseness=0,straight lines,numbering0=tangle},
tangle/.style={shape=tangle,numbering0=tangle},
every linkpattern/.style={x=\linkpatternunit,y=\linkpatternunit},
fused/.code={\linkpatternfused=#1},%
vertex/.style={circle,thin,solid,draw=black,fill=\linkpatternvertexcolor,inner sep=1.5pt,draw opacity=1,transform shape},%
edge/.style={very thick,solid,draw=\linkpatternedgecolor,draw opacity=1}}
\linkpatterncenteredfalse%
\linkpatterninvertedfalse%
\linkpatternnumberedfalse%
\linkpatterntikzstartedfalse%
\linkpatternboxedfalse%
\linkpatternaxistrue%
\linkpatternaliastrue%
\linkpatternstraightlinesfalse%
\def\linkpatternlooseness{0.2}%
\def\linkpatternsquareness{0.35}%
\def\linkpatternvertexcolor{red}%
\def\linkpatternedgecolor{black}%
\def\linkpatternboxcolor{none}%
\def\linkpatternheight{0}
\def\linkpatternwidth{0}
\def\linkpatternshape{default}
\def\linkpatternnumbering{default}
\def\linkpatternpos{(0,0)}
\def\linkpatternextraspace{0}
\linkpatternnodefalse
\def\firstchar#1#2\empty{#1}%
\def\linkpatterndo#1#2{
\edef\param{\csname linkpattern#2\endcsname}
\edef\firstcharparam{\expandafter\firstchar\param\empty}
\expandafter\ifcat\firstcharparam a
\expandafter\ifx\csname linkpattern#1\param\endcsname\relax
\csname linkpattern#1unknown\endcsname
\else
\csname linkpattern#1\csname linkpattern#2\endcsname\endcsname
\fi
\else
\csname linkpattern#1unknown\endcsname
\fi
}%
\def\linkpatterncoordtangle{\ifnum\x>\lphalfsize\pgfmathparse{\lpsize+1-\x}\xdef\lpcoordx{\pgfmathresult}\xdef\lpcoordy{\lpheight}\xdef\lpangle{270}\else\xdef\lpcoordx{\x}\xdef\lpcoordy{-\lpheight}\xdef\lpangle{90}\fi}
\def\linkpatterncoordpipedream{\ifnum\x>\lphalfsize\pgfmathparse{\lpsize+1-\x-0.5}\xdef\lpcoordx{\pgfmathresult}\xdef\lpcoordy{0}\xdef\lpangle{270}\else\pgfmathparse{0.5-\x}\xdef\lpcoordy{\pgfmathresult}\xdef\lpcoordx{0}\xdef\lpangle{0}\fi}
\def\linkpatterncoordrectangle{
\ifnum\x>\lptqsize
\pgfmathparse{\lpsize+1-\x-0.5}\xdef\lpcoordx{\pgfmathresult}\xdef\lpcoordy{0}\xdef\lpangle{270}
\else\ifnum\x>\lphalfsize
\pgfmathparse{\x-\lptqsize-0.5}\xdef\lpcoordy{\pgfmathresult}\xdef\lpcoordx{\linkpatternwidth}\xdef\lpangle{180}
\else\ifnum\x>\linkpatternheight
\pgfmathparse{\x-\linkpatternheight-0.5}\xdef\lpcoordx{\pgfmathresult}\xdef\lpcoordy{-\linkpatternheight}\xdef\lpangle{90}
\else
\pgfmathparse{0.5-\x}\xdef\lpcoordy{\pgfmathresult}\xdef\lpcoordx{0}\xdef\lpangle{0}
\fi\fi\fi
}%
\def\linkpatternsetsizeunknown{
\global\lpsize=\linkpatternsize%
\if\linkpatternheight0
\xdef\maxsep{0}
\foreach \x/\xx in \mylist%
{%
\edef\tempx{\withoutprime{\x}}
\edef\tempxx{\withoutprime{\xx}}
\pgfmathparse{max(\maxsep,abs(\tempx-\tempxx))}
\xdef\maxsep{\pgfmathresult}
}%
\pgfmathparse{0.25+0.8*\linkpatternsquareness*\maxsep}
\xdef\lpheight{\pgfmathresult}
\else
\xdef\lpheight{\linkpatternheight}
\fi
}
\newcount\tempsize
\def\linkpatternrightmostunknown{
\global\lpsize=0
\global\tempsize=0
\foreach\x/\labx in \linkpatternnumbering
{
\edef\tempx{\withoutprime{\x}}%
\ifnum\lpsize<\tempx\global\lpsize=\tempx\fi
\global\advance\tempsize by 1
}
\ifnum\tempsize>\lpsize\global\lpsize=\tempsize\fi
}%
\def\linkpatternrightmostdefault{
\global\lpsize=0
\global\tempsize=0
\foreach \x/\y in \mylist
{
\edef\tempx{\withoutprime{\x}}%
\ifnum\lpsize<\tempx\global\lpsize=\tempx\fi
\ifx\x\y
\global\advance\tempsize by 1
\else
\edef\tempy{\withoutprime{\y}}
\ifnum\lpsize<\tempy\global\lpsize=\tempy\fi%
\global\advance\tempsize by 2
\fi
}
\ifnum\tempsize>\lpsize\global\lpsize=\tempsize\fi
}%
\def\linkpatternrightmosttangle{
\global\lpsize=0
\global\tempsize=0
\foreach \x/\y in \mylist
{
\edef\tempx{\withoutprime{\x}}
\ifnum\lpsize<\tempx\global\lpsize=\tempx\fi
\ifx\x\y
\global\advance\tempsize by 1
\else
\edef\tempy{\withoutprime{\y}}
\ifnum\lpsize<\tempy\global\lpsize=\tempy\fi%
\global\advance\tempsize by 2
\fi
}
\global\advance\lpsize by\lpsize%
\ifnum\tempsize>\lpsize\global\lpsize=\tempsize\fi
}%

\newcommand\linkpattern[2][]{%
{%
\pgfkeys{/linkpattern/.cd,#1}%
\edef\mylist{#2}%
\def\primetest##1'{}%
\def\hasaprime##1{\expandafter\primetest##1''}%
\def\internalwithoutprime##1'{##1}%
\def\withoutprime##1{\if\hasaprime##1 %
\expandafter\internalwithoutprime##1\else ##1\fi}%
\iflinkpatternnumbered%
\iflinkpatterninverted%
\tikzset{/linkpattern/lbl/.style n args={3}{label={[/linkpattern/labeloptionslist=-##1,##3] ##1:##2}}}%
\else%
\tikzset{/linkpattern/lbl/.style n args={3}{label={[/linkpattern/labeloptionslist=##1,##3] ##1:##2}}}%
\fi%
\else%
\tikzset{/linkpattern/lbl/.style={}}%
\fi%
\tikzifinpicture{\linkpatterntikzstartedtrue%
\begin{scope}[shift=\linkpatternpos,/linkpattern/every linkpattern]%
}{%
\linkpatterntikzstartedfalse%
\iflinkpatterncentered
\begin{tikzpicture}[baseline=(current  bounding  box.center),/linkpattern/every linkpattern]%
\else
\begin{tikzpicture}[baseline=0,/linkpattern/every linkpattern]%
\fi
}%
\begin{scope}[local bounding box=link pattern box]
\iflinkpatterninverted%
\begin{scope}[yscale=-1]%
\fi%
\linkpatterndo{setsize}{shape}%
\ifnum\lpsize=0
\linkpatterndo{rightmost}{numbering}
\fi
\pgfmathtruncatemacro{\lphalfsize}{\lpsize/2}%
\linkpatterndo{numbering}{numbering}
\iflinkpatternboxed
\linkpatterndo{drawbox}{shape}
\else
\iflinkpatternaxis
\linkpatterndo{drawaxis}{shape}
\fi
\fi
\foreach\xx/\xlab/\opt in \lpnumbering
{
\ifx\xlab\opt\def\opt{}\fi%
\if\hasaprime\xx %
\pgfmathtruncatemacro{\xx}{\lpsize+1-\withoutprime{\xx}}
\fi
\ifnum\linkpatternfused>1
\pgfmathsetmacro{\x}{0.4*(0.5+\linkpatternfused*(0.5+floor((\xx-1)/\linkpatternfused)))+0.6*\xx}
\else
\def\x{\xx}
\fi
\linkpatterndo{coord}{shape}
\iflinkpatternalias\def\xlabb{\xlab}\else\def\xlabb{\xx}\fi
\path (\lpcoordx,\lpcoordy) coordinate[/linkpattern/vertex,/linkpattern/lbl={\lpangle+180}{\xlab}{\opt},alias=v\xlabb] (v\xx) ++(\lpangle:\linkpatternunit) coordinate[alias=vv\xlabb] (vv\xx);
}
\foreach \a/\b/\c in \mylist
{
\if\hasaprime\a %
\pgfmathtruncatemacro{\a}{\lpsize+1-\withoutprime{\a}}
\fi
\ifx\b\c\def\c{}\fi%
\draw[/linkpattern/edge]
\ifx\a\b
(v\a)
\c
--
++(0,\lpheight);
\else
\pgfextra{
\if\hasaprime\b %
\pgfmathtruncatemacro{\b}{\lpsize+1-\withoutprime{\b}}
\fi
\gettikzxy{(v\a)}{\ax}{\ay}
\gettikzxy{(v\b)}{\bx}{\by}
\gettikzxy{(vv\a)}{\axx}{\ayy}
\gettikzxy{(vv\b)}{\bxx}{\byy}
\pgfmathsetmacro{\dist}{sqrt((\ax-\bx)*(\ax-\bx)+(\ay-\by)*(\ay-\by))}
\pgfmathsetmacro{\abx}{(\axx-\ax)*\dist*\linkpatternsquareness+(\bx-\ax)*\linkpatternlooseness)}
\pgfmathsetmacro{\aby}{(\ayy-\ay)*\dist*\linkpatternsquareness+(\by-\ay)*\linkpatternlooseness)}
\pgfmathsetmacro{\bax}{(\bxx-\bx)*\dist*\linkpatternsquareness+(\ax-\bx)*\linkpatternlooseness)}
\pgfmathsetmacro{\bay}{(\byy-\by)*\dist*\linkpatternsquareness+(\ay-\by)*\linkpatternlooseness)}
}
(v\a)
\c%
\iflinkpatternstraightlines
\pgfextra{
\pgfmathsetmacro{\t}{((\ax-\bx)*\bay-(\ay-\by)*\bax)/(\aby*\bax-\abx*\bay)}
\pgfmathsetmacro{\abx}{\t*\abx}
\pgfmathsetmacro{\aby}{\t*\aby}
}
[rounded corners=0.2\linkpatternunit] -- ++(\abx,\aby) -- (v\b);
\else
.. controls ++(\abx,\aby) and ++(\bax,\bay) ..
\fi
(v\b);
\fi
}
\end{scope}
\iflinkpatternnode
\node[fit=(link pattern box),/linkpattern/nodeoptionslist] {};
\fi
\iflinkpatterninverted
\end{scope}
\fi
\iflinkpatterntikzstarted
\end{scope}
\else%
\end{tikzpicture}%
\fi%
}}%
\newcommand\tanglelinkpattern[3][]{%
{%
\pgfkeys{/linkpattern/.cd,#1}
\iflinkpatterninverted
\begin{tikzpicture}[/linkpattern/every linkpattern,baseline=\linkpatternunit]%
\else
\begin{tikzpicture}[/linkpattern/every linkpattern,baseline=-\linkpatternunit]%
\fi
\linkpattern[#1,tikzstarted,numbered=false]{#3}
\pgfmathtruncatemacro{\lptempsize}{2*\linkpatternsize}
\iflinkpatterninverted
\begin{scope}[yshift=0.5*\linkpatternunit]
\else
\begin{scope}[yshift=-0.5*\linkpatternunit]
\fi
\linkpattern[tangle,#1,tikzstarted,size=\lptempsize,
numbering=halftangle,
height=0.5]{#2}
\end{scope}
\end{tikzpicture}%
}}
\newcommand\diag[4][]{%
\pgfkeys{/linkpattern/.cd,#1}%
\iflinkpatterntikzstarted\else%
\begin{tikzpicture}[scale=0.5]%
\fi%
\iflinkpatterninverted%
\begin{scope}[yscale=-1]%
\fi%
\draw (0,0) grid (#2,#3);
\edef\mylist{#4}%
\foreach\y/\x/\z in \mylist
{
\ifx\x\z
\draw[decorate,decoration={zigzag,
amplitude=1pt,segment length=5pt}]
(\x-0.5,#3) -- (\x-0.5,\y-0.5) node[circle,fill=black,inner sep=2pt] {} -- (#2,\y-0.5);
\else
\node at (\x-0.5,\y-0.5) {$\z$};
\fi
}
\iflinkpatterninverted
\end{scope}
\fi
\iflinkpatterntikzstarted\else%
\end{tikzpicture}%
\fi%
}
\makeatletter
\tikzset{circle split part fill/.style  args={#1,#2}{%
 alias=tmp@name,
  postaction={%
    insert path={
     \pgfextra{%
     \pgfpointdiff{\pgfpointanchor{\pgf@node@name}{center}}%
                  {\pgfpointanchor{\pgf@node@name}{east}}%
     \pgfmathsetmacro\insiderad{\pgf@x}
      \fill[#1] (\pgf@node@name.base) ([xshift=-\pgflinewidth]\pgf@node@name.east) arc
                          (0:180:\insiderad-\pgflinewidth)--cycle;
      \fill[#2] (\pgf@node@name.base) ([xshift=\pgflinewidth]\pgf@node@name.west)  arc
                           (180:360:\insiderad-\pgflinewidth)--cycle;                    }}}}}
 \makeatother
\tikzset{bdot/.style={circle,circle split,draw,circle split part fill={black,white},thin,inner sep=1pt}}%
\tikzset{wdot/.style={circle,circle split,draw,circle split part fill={white,black},thin,inner sep=1pt}}%
\newdimen{\loopcellsize}
\tikzset{bgplaq/.style={draw=black,fill=\linkpatternboxcolor}}
\def\plaqwest{}
\def\plaqeast{}
\def\plaqnorth{}
\def\plaqsouth{}
\pgfkeys{/linkpattern/.cd,west/.store in=\plaqwest,east/.store in=\plaqeast,north/.store in=\plaqnorth,south/.store in=\plaqsouth}
\def\plaqname{plaq}%

\newcommand\plaq[2][]{
\node[bgplaq,rectangle,draw,minimum size=\loopcellsize,transform shape] (\plaqname) {};
\useasboundingbox;
\pgfkeys{/linkpattern/.cd,#1}%
\ifx#2\empty\else
\begin{scope}[x=\loopcellsize,y=\loopcellsize]
\csname plaq#2\endcsname
\end{scope}\fi
}
\newcommand\nbplaq[2][]{%
\node[bgplaq,rectangle,draw=none,fill=none,minimum size=\loopcellsize,transform shape] (\plaqname) {};
\useasboundingbox;
\pgfkeys{/linkpattern/.cd,#1}%
\ifx#2\empty\else
\begin{scope}[x=\loopcellsize,y=\loopcellsize]
\csname plaq#2\endcsname
\end{scope}\fi
}
\newcommand\ecplaq[3][]{%
\node[bgplaq,rectangle,draw,minimum size=\loopcellsize,transform shape] (\plaqname) {};
\useasboundingbox;
\pgfkeys{/linkpattern/.cd,#1}%
\begin{scope}[x=\loopcellsize,y=\loopcellsize]
\def\linkpatternedgecolor{#3}
\csname plaq#2\endcsname
\end{scope}
}
\newcommand\wecplaq[3][]{%
\node[bgplaq,rectangle,draw,fill=none,minimum size=\loopcellsize,transform shape] (\plaqname) {};
\useasboundingbox;
\pgfkeys{/linkpattern/.cd,#1}%
\begin{scope}[x=\loopcellsize,y=\loopcellsize]
\def\linkpatternedgecolor{#3}
\csname plaq#2\endcsname
\end{scope}
}
\newcommand\ecbcplaq[4][]{%
\node[bgplaq,rectangle,draw,fill=#4,minimum size=\loopcellsize,transform shape] (\plaqname) {};
\useasboundingbox;
\pgfkeys{/linkpattern/.cd,#1}%
\begin{scope}[x=\loopcellsize,y=\loopcellsize]
\def\linkpatternedgecolor{#3}
\csname plaq#2\endcsname
\end{scope}
}

\newcommand\wplaq[2][]{
\node[bgplaq,rectangle,draw,fill=none,minimum size=\loopcellsize,transform shape] (\plaqname) {};
\useasboundingbox;
\pgfkeys{/linkpattern/.cd,#1}%
\ifx#2\empty\else
\begin{scope}[x=\loopcellsize,y=\loopcellsize]
\csname plaq#2\endcsname
\end{scope}\fi
}

\tikzset{loop/.code={\def\plaqname{loop-\the\pgfmatrixcurrentrow-\the\pgfmatrixcurrentcolumn}},loop/.append style={matrix,row sep={\loopcellsize,between origins},column sep={\loopcellsize,between origins}}}%

\begin{document}

\title[Stretched Schubert coefficients are eventually quasi-polynomial]{Stretched
Schubert coefficients are \\ eventually quasi-polynomial}

\author[Igor Pak \. \and \. Zachary Slonim]{Igor Pak$^\star$  \. \and \.  Zachary Slonim$^\star$}

\thanks{\today}

\thanks{\thinspace ${\hspace{-.45ex}}^\star$Department of Mathematics,
UCLA, Los Angeles, CA 90095, USA. Email:  \texttt{\{pak,zachslonim\}@math.ucla.edu}}

\begin{abstract}
For a permutation \ts $u\in S_n$, let \ts $N\ast u\in S_{Nn}$ \ts be the permutation
with scaled Lehmer code.  For given \ts $u,v,w\in S_n$ \ts and integer \ts $N$, the
\emph{stretched Schubert coefficients}  \ts are defined as \ts $f_{u,v,w}(N):=c_{N*u,N*v}^{N*w}\ts$.
Our main result is that the function \ts $f_{u,v,w}(N)$ \ts is eventually quasi-polynomial.
This proves Kirillov's conjecture \cite[Conj.~6.28]{Kir04},
that the generating function for the sequence \ts $\{f_{u,v,w}(N)\}$ \ts is rational.

For the proof, we use combinatorics of pipe dreams to show that Schubert coefficients
are given as an alternating sum of the numbers of integer points in certain polytopes.
These polytopes behave nicely under stretching, and we use Ehrhart theory to obtain
the result.  As a consequence of the proof, we also present new counterexamples
to the saturation conjecture for Schubert coefficients, and give computational applications.
\end{abstract}

\maketitle

\section{Introduction}\label{s:intro}

\subsection{Foreword} \label{ss:intro-fore}
In algebraic combinatorics, the idea of \defna{stretching} %
has led to a series of profound insights into various
structure constants.  Notably, \defng{stretched Kostka numbers} \ts
$K_{N\la,N\mu}$ \ts and \defng{stretched Littlewood--Richardson $($LR$)$ coefficients}
 \ts $c^{N\la}_{N\mu,N\nu}$ \ts are polynomial in~$N$ \cite{DW02,Ras04}.
For general root systems, stretched tensor multiplicities are
\defn{quasi-polynomial} (polynomial on residue classes modulo a fixed integer),
see \cite{DM06}.

There are two profound insights giving motivation for the stretching.
First, in various contexts, Gelfand and coauthors suggested that
these constants count the number of integer points in polytopes given
by inequalities linear in the weights $\la,\mu,\nu$ \cite{GT50,GZ85}.
In this setting, stretched constants coincide with
\emph{Ehrhart $($quasi-$)$ polynomials}, see e.g.\ \cite{Bar08,BR07} and
a discussion in~$\S$\ref{ss:finrem-unim}.  This connection to Ehrhart theory has led to
a series of remarkable applications in algebraic combinatorics, see
e.g.\ \cite{BBCV06,Stu95}.

The second insight is the \defn{saturation property} \ts for  LR
coefficients, which
states that for all partitions \ts $\la,\mu,\nu$ \ts
with \ts $|\la|=|\mu|+|\nu|$, we have:
\begin{equation}\label{eq:saturation}
c^{\la}_{\mu,\nu} \. > \. 0 \quad \Longleftrightarrow \quad c^{N\la}_{N\mu, N \nu}
\. > \. 0 \ \ \text{\em for any \ $N\ge 1 \ts$}.
\end{equation}
Here \ts $N\. (\la_1,\la_2,\ldots):=(N\la_1,N\la_2,\ldots)$ \ts is the (usual)
stretching of partitions.
The saturation property \eqref{eq:saturation} was proved by Knutson and Tao \cite{KT99}.
It gives a complete resolution of the \emph{Horn problem} on spectra of Hermitian matrices,
see e.g.\ \cite{Buch00,Ful00,Kum14} for overviews of different aspects
of this remarkable story. %
This led to numerous developments in
combinatorics, representation theory and computational complexity,
see an overview in \cite{PR26}.

Motivated by these results, in 2004, Kirillov made a long series of conjectures
for various extensions of LR coefficients.  Notably, in \cite[Conj.~6.28]{Kir04}, he
asked  whether the generating function for the stretched Schubert coefficients
is always a rational function.  We prove this conjecture and discuss several
closely related problems by Kirillov on the stretched Schubert coefficients,
some already refuted and some still open.

\smallskip

\subsection{Main results} \label{ss:intro-main}
\defn{Schubert coefficients} \ts %
$\{c_{u,v}^w \. : \. u,v,w\in S_\infty\}$ \ts
can be defined as structure constants for \emph{Schubert polynomials}:
\begin{equation}\label{eq:Schub-def}
\Sc_u\. \cdot\.\Sc_v\ = \ \sum_{w\ts\in\ts S_\infty}\. c_{u,v}^w\. \Sc_w\..
\end{equation}
Here $S_\infty$ consists of permutations with all but finitely many fixed points,
and \ts $\Sc_u \in \nn[x_1,x_2,\ldots]$.

It is known that \ts $c_{u,v}^w$ \ts are nonnegative integers as they
count certain intersection numbers.
These coefficients play a major role in \defng{Schubert calculus}, a rapidly developing
area of algebraic geometry, motivated in part by rich connections with
representation theory and algebraic combinatorics, see e.g.\ \cite{AF24,Knu22}.

For a permutation $w\in S_n\ts$, the \defn{Lehmer code}, also called
the \emph{inversion index}, is defined as
$$
\code(w) \, := \, (c_1,\ldots,c_n)\in \nn^n, \quad \text{where} \quad c_i \. := \.
|\{j\,: \, j>i, \. w(j)<w(i)\}|.
$$
Clearly, \ts $c_1+\ldots+c_n = \inv(w)$ \ts is the \emph{number of inversions}.
\defn{Stretched permutations} \ts are defined as
\[
N\satA{} w \, := \, \code^{-1}(N c_1,N c_2,\ldots,N c_n,0,\ldots,0) \ts \in \ts S_{Nn}\..
\]
Finally, the \defn{stretched Schubert coefficients} \ts are defined as \ts $c_{N*u,N*v}^{N*w}\ts$.

For \emph{Grassmannian permutations} (permutations with one descent),
this definition coincides with scaled partitions (see~$\S$\ref{ss:back-Grass}).  Thus, \defna{Kirillov's conjecture}
\cite[Conj.~6.28]{Kir04}, can be viewed as an extension of the Derksen--Weyman \cite[Cor.~3]{DW02}
and Rassart \cite[Cor.~4.2]{Ras04} results (see also \cite{Tha25}), on polynomiality of stretched LR coefficients:

\smallskip

\begin{thm}[{\rm formerly Kirillov's conjecture}{}]\label{t:main}
For all \ts $u,v,w\in S_n$, the generating function
\begin{equation}\label{eq:gf}
P_{u,v,w}(t) \, := \, \sum_{N=1}^\infty \. c_{N*u,N*v}^{N*w} \, t^N
\end{equation}
is rational.  Moreover, the function \ts $f_{u,v,w}(N) := c_{N*u,N*v}^{N*w}$ \ts
is eventually quasi-polynomial.
\end{thm}

\smallskip

Here a function \ts $f:\nn_{\ge 1} \to \nn$ \ts is \defn{eventually polynomial} \ts
if there exists an integer $M$ and a polynomial~$g$, such that \ts $f(N)=g(N)$ \ts
for all \ts  $N\ge M$.  Similarly, $f$ is \defn{eventually quasi-polynomial},
if there exists integers \ts $m, M$ \ts and polynomials \ts $g_i$ for \ts $0\le i \le m-1$
such that \ts $f(N)=g_i(N)$ \ts for all $N\ge M$ \ts and \ts $N=i \! \mod m$.

\smallskip

\subsection{Prior work and related results} \label{ss:intro-rel}
Theorem~\ref{t:main} is a rare positive result in this direction, and is best understood
in the context of closely related negative results on Schubert coefficients (see below)
and Kronecker coefficients (see~$\S$\ref{ss:finrem-Kron}).

First, in the same paper (\cite[Conj.~6.28]{Kir04}),
Kirillov further conjectured the saturation property for the
stretched Schubert coefficients:
\begin{equation}\label{eq:saturation-Sch}
\qquad c^{w}_{u,v} \. > \. 0 \quad \Longleftrightarrow \quad c^{N\ast w}_{N\ast u, N \ast v}
\. > \. 0 \ \ \text{\em for any \ $N\ge 1$}.
\end{equation}
In \cite{PR26}, the first author and Robichaux disproved \eqref{eq:saturation-Sch}
with a family of examples with \ts $c^w_{u,v}=1$ \ts and \ts $c^{N\ast w}_{N\ast u, N \ast v}=0$
\ts for all \ts $N\ge 2$.  Their smallest example has \ts $u,v,w\in S_4$ (cf.~$\S$\ref{ss:finrem-Kostka}).
This implies that \ts $f_{u,v,w}$ \ts cannot be (quasi-) polynomial, so the \emph{eventual}
\ts (quasi-) polynomiality as in the theorem is the best one could hope for
(see also~$\S$\ref{ss:finrem-conj}).

Second, Kirillov also asked in \cite[$\S$6.9]{Kir04}, whether the (quasi-)
polynomials \ts $f_{u,v,w}$ \ts always have nonnegative coefficients.\footnote{Kirillov
writes in \cite[$\S$6.9]{Kir04}, that this question was not ``extensively tested''. }
For the stretched Kostka and LR coefficients, this is a well-known conjecture
by King, Tollu and Toumazet \cite[$\S$3]{KTT04}, see also \cite{DM06}.  However,
the following example shows that this question by Kirillov
has a negative answer, i.e., there are stretched Schubert coefficients
given by a polynomial which has a negative coefficient:

\smallskip

\begin{prop}%
\label{p:negex}
Take \. $u=3142$, $v=1432$,
$w=4321\in S_4\ts$.  Then:
$$
c_{N*u,N*v}^{N*w} \. = \. N-1 \quad \text{for all} \quad N\ge 1.
$$
\end{prop}

\smallskip

Third, famously, Schubert coefficients are not known to have a combinatorial
interpretation, %
see e.g.\ \cite{Knu22,Pak-OPAC,Sta00}.  Even when such combinatorial interpretations
are known, for example for permutations with at most two or three descents
\cite{BKPT16,KZ17}, these are given in terms of \emph{puzzles}
(certain tilings of triangular shapes, see also~$\S$\ref{ss:finrem-puzzle}),
and not integer points in convex polytopes.  The same holds for
known \emph{signed combinatorial interpretations} \cite{PR-puzzle,PS09}.
This rules out a direct geometric approach to Theorem~\ref{t:main}.

Finally, a word of caution.  While the stretched Schubert coefficients
are eventually quasi-polynomial, it may take awhile until this
regular behavior starts.  The following result shows that the starting
point \ts $M$ \ts can be linear in~$n$.

\smallskip

\begin{thm}%
\label{t:offset}
Let $n\ge 7$ and let \ts $u,v,w\in S_n$ be given by:
$$u=v=(2,1,3,4,5,\ldots,n-2, n,n-1) \qquad \text{and} \qquad
w=(4,1,2,3,5,\ldots,n-2, n,n-1).
$$
Then:
$$
c_{N*u,N*v}^{N*w} \ts = \ts 1 \ \ \text{for} \ \ N \ts = \ts 2n-6 \qquad \text{and} \qquad
c_{N*u,N*v}^{N*w} \ts =\ts 0 \ \ \text{for all} \ \ N \ts \ge \ts 2n-5.
$$
\end{thm}

\smallskip

Note that both Proposition~\ref{p:negex} and Theorem~\ref{t:offset}
refute the saturation property \eqref{eq:saturation-Sch}.  %
Theorem~\ref{t:offset} refutes the \ts $\Rightarrow$ \ts implication,
the same as the examples given in~\cite{PR26}.
For the saturation property \eqref{eq:saturation} of LR coefficients,
this is the ``easy direction'' known for decades, see e.g.\ \cite{Zel99}.
On the other hand, Proposition~\ref{p:negex} refutes the implication~$\ts\Leftarrow$.
For the saturation property \eqref{eq:saturation} of LR coefficients,
this is the ``hard direction'' proved in~\cite{KT99}.

\smallskip

\subsection{Proof idea} \label{ss:intro-proof}
In our proof of Theorem~\ref{t:main}, we show that Schubert
coefficients are given as an alternating sum of the numbers of integer
points in polytopes given by a large system of linear inequalities.
We believe this is the first result of this type, as other notable
general formulas involve alternating sums over chains in the Bruhat order
\cite{Bil99,PS09,PR-signed} or puzzles \cite{PR-puzzle}.

We then show that for stretched Schubert coefficients the polytopes as above
are defined by the same linear inequalities, with the parameter~$N$ only
in the constraints.  We then use a known result, Lemma~\ref{l:quasipoly} by Chen, Lil
and~Sam \cite{CLS12},  to conclude the eventual quasi-polynomiality.  This is
a well-known phenomenon established in a much larger geometric context, see
\cite{Shen18,Woo14}.

Let us emphasize that our general approach is effective and gives an upper
exponential bound \ts $M\le 2^{3n^4}$ \ts for the starting point,
and a doubly exponential upper bound \ts $m \le \big((3 n)^{3 n^2}\big)!$ \ts
for the period, see Theorem~\ref{t:quant}.  We did not try to
optimize these bounds, which likely can be significantly improved.
In fact, our experiments suggest that \ts $f_{u,v,w}(N)$ \ts
are eventually polynomial (so have period $m=1$),
with a linear starting point, see~$\S$\ref{ss:finrem-conj}.

\smallskip

\subsection{Computational applications} \label{ss:intro-cs}
For a permutation \ts $w\in S_n$, let \ts
$\Exc(w):= \{i\.:\. w(i)>i\}$ \ts denote the set of \defn{exceedances},
and note that \ts $\code_i(w)=0$ \ts for all \ts $i \notin \Exc(w)$.
Let
$$
\Range(w) \,:= \, [\min \Exc(w), \max \Exc(w)] \, \subseteq \, \{1,2,\ldots,n-1\}
$$
denote the \defn{range} of~$w$.   We consider the problem of computing
Schubert coefficients for permutations of bounded range.  In this case,
permutations are defined by the restriction of the code to the range:
$$
\big\{\code_i(w) \. : \. i\in \Range(w)\big\}.
$$
This is a typical setting in \emph{parameterized complexity}, cf.~\cite{C+15},
which allows us to skip all zeroes in the Lehmer code as part of the input.

\smallskip

\begin{thm}\label{t:cs}
Let \ts $u,v,w\in S_n$ \ts be three permutations with a bounded total range
\begin{equation}\label{eq:boundedrange}
\psi \ := \    \max\ts\{\Exc(u),\Exc(v),\Exc(w)\} \. - \. \min\ts\{\Exc(u),\Exc(v),\Exc(w)\} \. + 1\ts.
\end{equation}
We assume that the permutations \ts $u,v,w$ \ts are given by the restrictions of
their Lehmer codes to the range.  Then
the Schubert coefficient \ts $c^w_{u,v}$ \ts can be computed in time
$$e^{O(\psi^{2.5})} \. (\log n)^{O(\psi^2\log \psi)}\..
$$
\end{thm}

\smallskip

Note that the bit-length of code entries is $O(\log n)$, so
the total input size of permutations \ts $u,v,w$ \ts as in the theorem,
is \ts $O(\psi \log n)$.  This implies that for every fixed~$\psi$, the theorem
gives a polynomial time algorithm.  In the opposite extreme,
when \ts $\psi=n$, this bound is not too far off the easy \ts $2^{O(n^2)}$ \ts
bound given by exhaustive enumeration of all pipe dreams in the
Postnikov--Stanley formula \eqref{eq:post-stan}.

Theorem~\ref{t:cs} is the first rigorous algorithmic result for Schubert
coefficients.  We refer to
\cite{BV08,HS17} for theoretical and computational approaches to
computing Schubert coefficients, but without running time analysis.
In computational complexity, it is known that computing Schubert
coefficients is in \ts $\GapP=\SP-\SP$, see \cite[$\S$10]{Pak-OPAC},
but it is open whether the problem is $\SP$-hard, see \cite[Conj.~1.2]{PR-vanish}.
By contrast, the vanishing problem \ts $\{c^w_{u,v}=^?0\}$ \ts is in $\BPP$,
i.e.\ can be solved in probabilistic polynomial time \cite{PR-BPP}.  Finally,
we note that using the Lehmer code to encode permutations is not unusual
in the context of Schubert polynomials, see e.g.\ \cite[$\S$2.1]{Man01}
and \cite[$\S$6]{MPY22}.

The proof of Theorem~\ref{t:cs} is based on
the proof of Theorem~\ref{t:main} and %
the celebrated \defng{Barvinok's algorithm} \ts for counting the number
of integer points in convex polytopes of bounded dimension,
see \cite{Bar08,BP99} for details.

\smallskip

\subsection{Structure of the paper} \label{ss:intro-structure}
We start with the Schubert calculus background in Section~\ref{s:back},
where we include standard definitions, notation and basic results in the area.
In Section~\ref{s:defs}, we show how pipe dreams can be uniquely encoded
by certain ladder sequences.  In a technical Section~\ref{s:bijection},
we prove a complete set of inequalities such ladder sequences must
satisfy.  We give the proof of Theorem~\ref{t:main} and its effective
version (Theorem~\ref{t:quant}) in Section~\ref{s:proof}.
In Section~\ref{s:ex}, we prove Propositions~\ref{p:negex} and
Theorem~\ref{t:offset} giving explicit examples of stretched Schubert
coefficients (cf.~$\S$\ref{ss:finrem-puzzle}).  We conclude with final remarks and open problems
in Section~\ref{s:finrem}.

\medskip

\section{Basic definitions and background on Schubert polynomials}\label{s:back}

\subsection{Basic notation}\label{ss:back-basic}
We use \ts $\nn=\{0,1,2,\ldots\}$, \ts $\nn_{\ge 1}=\{1,2,\ldots\}$ \ts and \ts $[n]=\{1,2,\ldots,n\}$.  We use \ts $\kk$ \ts  to denote the set of infinite sequences with entries in $\nn$ with finite support.  When writing such sequences, we omit the infinite tail of zeros, and write only the prefix with the support of the sequence, so e.g.\ \ts $30120000\ldots$ \ts is written as \ts $3012$.

We use $S_n$ to denote the symmetric group, which we view as the group of
permutations of~$[n]$. Denote by \ts $\iota$ \ts the inclusion $\iota: S_n \. \hookrightarrow \. S_{n+1}$ \ts defined by $w(1)\cdots \. w(n) \, \mapsto \, w(1) \ \cdots \. w(n) \,\. n+1$.
As above, let \ts $S_\infty=\bigcup_{n\geq 1} S_n$ \ts
denote the group of permutations of \ts $\nn_{\ge 1}=\{1,2,\ldots\}$ \ts with all but finitely many fixed points, where the inclusion is given by~$\iota$. By analogy with infinite sequences, when writing such permutations we omit the tail of fixed points, so e.g.\ $34125678\ldots$ \ts is written as $3412$.

For convenience, we use both the sequence and word notation for permutations, so for example \ts $(3,1,4,2)$ \ts and \ts $3142$ \ts correspond to the same permutation in~$S_4\ts$.
For a permutation $w\in S_n\ts$, let \ts $\ell(w):=|\{(i,j)\.:\. 1\le i<j\le n, \. w(i)>w(j)\}|$ \ts denote the \defn{number of inversions} \ts in~$w$. Finally, let \ts $\Des(w):= \{i\.:\. w(i)>w(i+1)\}$ \ts denote the \defn{set of descents} \ts in~$w$. A permutation \ts $w\in S_n$ \ts is \defn{Grassmannian} \ts if it has at most one descent.

\smallskip

\subsection{Schubert coefficients}\label{ss:back-Sch}
Below we give a brief reminder of few basic results on Schubert polynomials.
We refer to \cite{Knu16,Mac91,Man01} for standard introductions to combinatorial
aspects, and to \cite{AF24,Ful97} for geometric aspects.

\defn{Schubert polynomials} \ts give a graded $\mathbb{Z}$-linear basis of polynomials ${\mathbb Z}[x_1,x_2,\ldots]$,
which we define recursively. Let \. $w_\circ=(n,n-1,\ldots,1) \in S_n\ts$.  Define
\begin{align*}
    {\mathfrak S}_{w_\circ}(x_1,\ldots,x_n)&:=x_1^{n-1}x_2^{n-2}\cdots x_{n-1}\,, \ \ \text{ and }\\
    {\mathfrak S_w}(x_1,\ldots,x_n)&:=\partial_i {\mathfrak S}_{ws_i}(x_1,\ldots,x_n) \ \ \text{ if } \ \ w(i)<w(i+1),
\end{align*}
where
$$\partial_i f \ := \ \frac{f-s_if}{x_i-x_{i+1}}\,.$$
Note that under the inclusion \.
$\iota: S_n\hookrightarrow S_{n+1}$ \. we have \.
${\mathfrak S}_w={\mathfrak S}_{\iota(w)}$.
This allows us to consider ${\mathfrak S}_w$ for each $w\in S_{\infty}$.
We do not need this algebraic definition, and instead only use a combinatorial
definition in terms of pipe dreams (see Theorem~\ref{t:schub}).

As the Schubert polynomials \ts $\{{\mathfrak S}_w\}$ \ts form a polynomial basis, we can define the \defn{Schubert coefficients}
(also called \emph{Schubert structure coefficients}) \ts
$\{c_{u,v}^w\}$ \ts by \eqref{eq:Schub-def}.
As we mentioned in the Introduction, although %
Schubert coefficients are nonnegative integers, 
they have no known combinatorial interpretation in full generality.

\smallskip

\smallskip

\subsection{Lehmer codes}\label{sec:encode}
For a permutation \ts $w\in S_\infty\ts$, its \emph{Lehmer code}  \ts
is the vector
\[\code(w)\, := \, (c_1(w),c_2(w),\ldots) \in \kk,\]
where \ts $c_i(w)$ \ts is the number
\[c_i(w)\, := \, |\{j\in [n]\. : \.  j>i, w(i)>w(j)\}|,
\]
defined as in the introduction.
We shorten this to ``code'' when the
context is clear.
Note that the map \ts $\code: S_\infty \to \kk$ \ts is a bijection, i.e.\
\ts $\code(w)$ \ts uniquely determines \ts $w\in S_{\infty}$ \ts and \ts $\code^{-1}$ \ts
is well defined.

\smallskip

\subsection{Code scaling}
\label{ss:Kir-prelim}
Let \ts $w\in S_n$ \ts be a permutation with \ts
$\code(w)=(c_1,\ldots,c_n)$.  Recall that for a given integer \ts
$N\ge 1$, the \defn{code scaling} \ts $N\ast w$ \ts is the
unique permutation with code \ts $(N c_1,\ldots,N c_n)$.

Observe that \ts $\ell(N\ast w)= N \ell(w)$.  Thus, we have
\ts $\ell(u)+\ell(v)=\ell(w)$ \ts implies \ts
$\ell(N\satA{}u)+\ell(N\satA{}v)=\ell(N\satA{}w)$ \ts for all $N>1$.
Additionally, it is clear by construction that code scaling preserves
the underlying descent sets of the permutations:  \ts $\Des(N\ast w) = \Des(w)$.

\smallskip

\subsection{Grassmannian permutations} \label{ss:back-Grass}
A permutation is called \emph{Grassmannian} \ts if it has at most one descent. It is  easy to see that the code of a Grassmannian
permutation $w$ is a partition $\la(w)$ written in reverse,
so this code scaling corresponds to the usual multiplication
of partitions by a constant~$N$, as in~$\S$\ref{ss:intro-fore}.

Moreover, the corresponding Schubert polynomials are
symmetric and coincide with Schur polynomials: \ts $\Sch_w=s_{\la(w)}$.
Thus, Schubert coefficients can be viewed as advanced generalizations
of LR~coefficients, see e.g.\ \cite[$\S$2.6.4]{Man01}. This relates Kirillov's conjecture proved in Theorem~\ref{t:main}
with polynomiality of LR coefficients as mentioned in~$\S$\ref{ss:intro-main}.

\smallskip

\subsection{Pipe dreams}\label{ss:pd} Let \ts $u\in S_n\ts$.  Define
\[\Delta_n:=\{(i,j)\in [n]\times [n]\. : \.  i+j \le n+1\},\]
and let $D\subseteq \Delta_n$. Visualize $D$ as a triangular grid of
square tiles using matrix indexing. We place an \defn{elbow tile} \ts
in positions $(i,j)\in \Delta_n-D$ and a \defn{cross tile} \ts in positions $(i,j)\in D$ (see Figure~\ref{fig:tiles}).
\newdimen{\loopcellsize}\setlength{\loopcellsize}{0.8cm}

\begin{figure}[hbt]
\begin{tikzpicture}[scale=0.5]
    \node[loop]{\wplaq{a} &&&& \wplaq{c}\\};
\end{tikzpicture}\caption{\label{fig:tiles}Elbow tile (left) and cross tile (right)}
\end{figure}
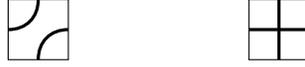

Doing this will connect the top boundary with the left boundary by pipes. Label the pipes entering the top boundary $1,2,\ldots, n$ \ts from left to right. Then, we say that $D$ is a \defn{pipe dream for $u$} if no two curves (\emph{pipes}) cross each other more than once, and if we read the labels of the pipes that come out the left edge top to bottom, we get the permutation $u$.
\newdimen{\loopcellsize}\setlength{\loopcellsize}{0.5cm}

\begin{ex}\label{ex:pd}
    Let $u=14862357\in S_8$ and let
    \[D:=\{(1,3),(1,5),(2,1),(2,3),(2,5),(2,6), (3,1),(3,2),(3,3),(4,1)\}\subseteq\Delta_8.\]
    Then $D$ can be visualized as the pipe dream on the left of Figure \ref{fig:expd}. We will use this $D$ as a running example throughout the paper.

    \begin{figure}[hbt]
\begin{minipage}{.45\textwidth}\centering

        \begin{tikzpicture}[scale=0.5]

\node[loop]{
\wplaq{a}&\plaq{a}&\plaq{c}&\plaq{a}&\wplaq{c}&\wplaq{a}&\wplaq{a}&\plaq{a}\\
\plaq{c}&\wplaq{a}&\wplaq{c}&\wplaq{a}&\plaq{c}&\plaq{c}&\plaq{a}\\
\plaq{c}&\wplaq{c}&\wplaq{c}&\wplaq{a}&\plaq{a}&\plaq{a}\\
\plaq{c}&\wplaq{a}&\wplaq{a}&\wplaq{a}&\plaq{a}\\
\wplaq{a}&\plaq{a}&\plaq{a}&\plaq{a}\\
\wplaq{a}&\plaq{a}&\plaq{a}\\
\wplaq{a}&\plaq{a}\\
\plaq{a}\\
};
\node at (-3.5,4.45) {$1$};
\node at (-2.5,4.45) {$2$};
\node at (-1.5,4.45) {$3$};
\node at (-0.5,4.45) {$4$};
\node at (0.5,4.45) {$5$};
\node at (1.5,4.45) {$6$};
\node at (2.5,4.45) {$7$};
\node at (3.5,4.45) {$8$};
\node at (-4.45,3.5) {$1$};
\node at (-4.45,2.5) {$4$};
\node at (-4.45,1.5) {$8$};
\node at (-4.45,0.5) {$6$};
\node at (-4.45,-0.5) {$2$};
\node at (-4.45,-1.5) {$3$};
\node at (-4.45,-2.5) {$5$};
\node at (-4.45,-3.5) {$7$};
\end{tikzpicture}
\end{minipage}
\qquad
\begin{minipage}{.45\textwidth}\centering
\vspace{-10px}
        \begin{tikzpicture}[scale=0.5]
\node at (0,0.79) {$1$};
\node at (1,0.79) {$2$};
\node at (2,0.79) {$3$};
\node at (3,0.79) {$4$};
\node at (4,0.79) {$5$};
\node at (5,0.79) {$6$};
\node at (6,0.79) {$7$};
\node at (7,0.79) {$8$};

\node at (-1,0) {$1$};
\node at (-1,-1) {$4$};
\node at (-1,-2) {$8$};
\node at (-1,-3) {$6$};
\node at (-1,-4) {$2$};
\node at (-1,-5) {$3$};
\node at (-1,-6) {$5$};
\node at (-1,-7) {$7$};

\node at (0,0) {$\cdot$};
\node at (1,0) {$\cdot$};
\node at (2,0) {$+$};
\node at (3,0) {$\cdot$};
\node at (4,0) {$+$};
\node at (5,0) {$\cdot$};
\node at (6,0) {$\cdot$};
\node at (7,0) {$\cdot$};

\node at (0,-1) {$+$};
\node at (1,-1) {$\cdot$};
\node at (2,-1) {$+$};
\node at (3,-1) {$\cdot$};
\node at (4,-1) {$+$};
\node at (5,-1) {$+$};
\node at (6,-1) {$\cdot$};

\node at (0,-2) {$+$};
\node at (1,-2) {$+$};
\node at (2,-2) {$+$};
\node at (3,-2) {$\cdot$};
\node at (4,-2) {$\cdot$};
\node at (5,-2) {$\cdot$};

\node at (0,-3) {$+$};
\node at (1,-3) {$\cdot$};
\node at (2,-3) {$\cdot$};
\node at (3,-3) {$\cdot$};
\node at (4,-3) {$\cdot$};

\node at (0,-4) {$\cdot$};
\node at (1,-4) {$\cdot$};
\node at (2,-4) {$\cdot$};
\node at (3,-4) {$\cdot$};

\node at (0,-5) {$\cdot$};
\node at (1,-5) {$\cdot$};
\node at (2,-5) {$\cdot$};

\node at (0,-6) {$\cdot$};
\node at (1,-6) {$\cdot$};

\node at (0,-7) {$\cdot$};

\end{tikzpicture}
\end{minipage}

\caption{\label{fig:expd}Pipe dream $D$ for $u$ with pipes drawn in (left) and simplified version with only dots and crosses (right)}
\end{figure}
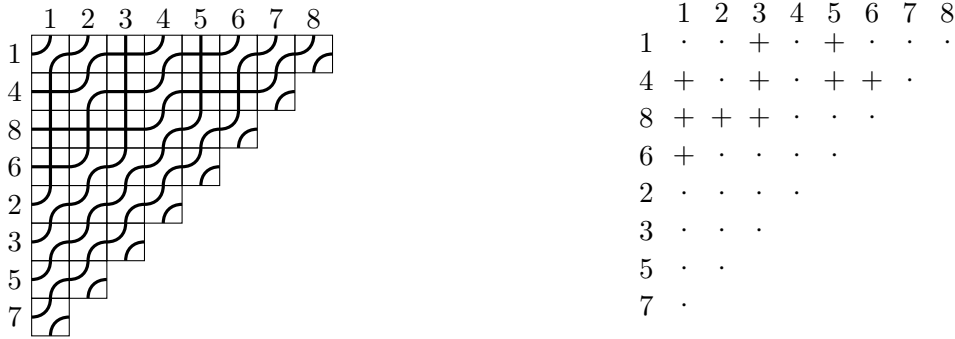
\end{ex}

As shown on the right of Figure~\ref{fig:expd}, we can display the information of a pipe dream by replacing each elbow tile with a \ts $\cdot$ \ts and each cross tile with a \ts $+$. For ease of reading, we will adapt this notation for the rest of the paper, while still referring to such configurations as pipe dreams. We will then refer to elbow tiles as ``dots'' and cross tiles as ``crosses''.

Given a pipe dream $D$, we define a \defn{patch of crosses} to be a subset $P\subseteq D$ such that
\[P=\{(i,j), (i,j+1),\ldots, (i,j+k-1)\},\]
for some $i,j,k\in \nn_{\ge 1}$. The \defn{size} of the patch $P$ is simply $|P|=k$.

Denote by $\pd(u)$ the set of pipe dreams for $u$. The condition that no two pipes cross more than once is often referred to in the literature as being \emph{reduced}, but in this paper, we will assume all pipe dreams are reduced. The \defn{weight} of a pipe dream $D$ for $u\in S_n$ is the tuple $\wtx(D):=(\cw_1,\cw_2,\ldots)$, where $\cw_i$ is the number of crosses in row $i$ of $D$. In Example~\ref{ex:pd}, $\wtx(D)=(2,4,3,1)$.

The following celebrated result gives a combinatorial definition of Schubert polynomials:

\smallskip

\begin{thm}[{\rm Bergeron--Billey \cite[Cor.~3.3]{BB93}}{}]\label{t:schub}
    Let \ts $u\in S_n$. Then:
    \begin{equation}\label{eq:schubpd}
        \Sch_u(x) \ = \ \sum_{D\ts\in \ts\pd(u)} \. x^{\wtx(D)}.
    \end{equation}
\end{thm}

\smallskip

For \ts $a=(a_1,a_2,\ldots)$, denote by \ts $K_{u,a}=[x^a]\Sch_u$ \ts the
\defn{Schubert--Kostka number}.
Theorem~\ref{t:schub} gives a combinatorial interpretation for \ts $K_{u,a}$ \ts
as the number of pipe dreams with \ts $a_i$ \ts crosses in the $i$-th row for each $1\le i\le n$.
The following result gives a signed combinatorial interpretation for
Schubert coefficients:

\smallskip

\begin{thm}[{\rm Postnikov--Stanley \cite[Cor.~17.13]{PS09}}{}]\label{t:kostasum}
    For $u,v,w\in S_n$, we have
    \begin{equation}\label{eq:post-stan}
    c_{u,v}^w\. = \. c_{u,v,w_{\circ}w} \, = \, \sum_{\sigma \in S_n}\sign (\sigma)\sum_{(a,b,c)\ts\in\ts \Omega(\sigma)} K_{u,a}\. K_{v,b}\. K_{w_{\circ}w,c}\,,
    \end{equation}
    where \. $\Omega(\sigma):=\{(a,b,c)\in (\nn^n)^3\. : \.  a+b+c=\sigma\rho_n\}$  \. and \. $\rho_n:=(n-1,n-2,\ldots,1,0)$.
\end{thm}

\smallskip

\subsection{Ladder moves}\label{ss:ladder} For $u\in S_n$ and $D\in \pd(u)$, a \defn{ladder move} (defined in \cite{BB93}) $L_{i,j}$ is a transformation of the type shown going left to right in Figure~\ref{fig:laddermove}.
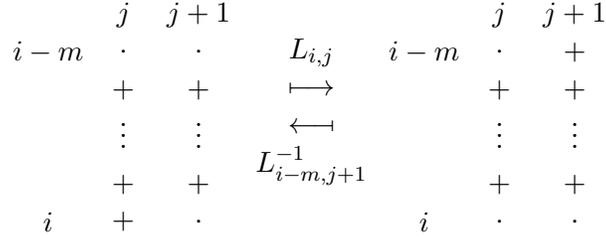
\begin{figure}[hbt]
    \begin{tikzpicture}[scale=0.5]

\node at (0,1) {$j$};
\node at (2,1) {$j+1$};
\node at (-2,0) {$i-m$};
\node at (-2,-4.5) {$i$};
\node at (0,0) {$\cdot$};
\node at (0,-1) {$+$};
\node at (0,-2) {$\vdots$};
\node at (0,-3.5) {$+$};
\node at (0,-4.5) {$+$};
\node at (2,0) {$\cdot$};
\node at (2,-1) {$+$};
\node at (2,-2) {$\vdots$};
\node at (2,-3.5) {$+$};
\node at (2,-4.5) {$\cdot$};

\node at (5,0) {$L_{i,j}$};
\node at (5,-1) {$\longmapsto$};

\node at (5,-3) {$L_{i-m,j+1}^{-1}$};
\node at (5,-2) {$\longmapsfrom$};

\node at (10,1) {$j$};
\node at (12,1) {$j+1$};
\node at (8,0) {$i-m$};
\node at (8,-4.5) {$i$};
\node at (10,0) {$\cdot$};
\node at (10,-1) {$+$};
\node at (10,-2) {$\vdots$};
\node at (10,-3.5) {$+$};
\node at (10,-4.5) {$\cdot$};
\node at (12,0) {$+$};
\node at (12,-1) {$+$};
\node at (12,-2) {$\vdots$};
\node at (12,-3.5) {$+$};
\node at (12,-4.5) {$\cdot$};
\end{tikzpicture}

\caption{\label{fig:laddermove}Ladder move}
\end{figure}

Formally, let \.  $L_{i,j}(D):=D+(i-m, j+1)-(i,j)$, where the following conditions are satisfied
\begin{itemize}
    \item [(a)] \ \  $(i,j)\in D$ \ts and \ts $(i,j+1)\notin D$,
    \item [(b)] \ \ $(i-m,j), (i-m,j+1)\notin D$ \ts for some \ts $0<m<i$, \ts and
    \item [(c)] \ \ $(i-k,j),(i-k,j+1)\in D$ \ts for all \ts $1\le k <m$.
\end{itemize}

Similarly, an \defn{inverse ladder move} \. $L_{i-m,j+1}^{-1}$ \. is a transformation of the type shown going right to left in Figure~\ref{fig:laddermove}. We think of these transformations as moving exactly one cross. Note that we index the (inverse) ladder move by the row and column that this cross starts in. Thus, we can see that \ts $L_{i-m,j+1}^{-1}=(L_{i,j})^{-1}$, so \ts $L_{i-m,j+1}^{-1}\circ L_{i,j}=\id$.
\vspace{4px}

For $u\in S_n$, let $D_\Bot(u)\in \pd(u)$ be the pipe dream for $u$ whose crosses are all left-aligned and with weight $\wtx(D_\Bot(u))=\code(u)$, as in Example \ref{ex:dbot}. We call $D_\Bot(u)$ the \defn{bottom pipe dream for~$u$}. It is well-known that $D_\Bot(u)$ is indeed a pipe dream for~$u$ and has the largest weight among pipe dreams for $u$ in the reverse lexicographic order, see~\cite{BB93}.

\begin{ex}\label{ex:dbot}
    For $u=14862357\in S_8\ts$, we have $\code(u)=(0,2,5,3)$.  In this case,
    the bottom pipe dream \ts $D_\Bot(u)$ \ts for~$u$ is shown in Figure~\ref{fig:dbot}.
        \begin{figure}[hbt]
    \begin{tikzpicture}[scale=0.5]
\node at (0,1) {$1$};
\node at (1,1) {$2$};
\node at (2,1) {$3$};
\node at (3,1) {$4$};
\node at (4,1) {$5$};
\node at (5,1) {$6$};
\node at (6,1) {$7$};
\node at (7,1) {$8$};

\node at (-1,0) {$1$};
\node at (-1,-1) {$4$};
\node at (-1,-2) {$8$};
\node at (-1,-3) {$6$};
\node at (-1,-4) {$2$};
\node at (-1,-5) {$3$};
\node at (-1,-6) {$5$};
\node at (-1,-7) {$7$};

\node at (0,0) {$\cdot$};
\node at (1,0) {$\cdot$};
\node at (2,0) {$\cdot$};
\node at (3,0) {$\cdot$};
\node at (4,0) {$\cdot$};
\node at (5,0) {$\cdot$};
\node at (6,0) {$\cdot$};
\node at (7,0) {$\cdot$};

\node at (0,-1) {$+$};
\node at (1,-1) {$+$};
\node at (2,-1) {$\cdot$};
\node at (3,-1) {$\cdot$};
\node at (4,-1) {$\cdot$};
\node at (5,-1) {$\cdot$};
\node at (6,-1) {$\cdot$};

\node at (0,-2) {$+$};
\node at (1,-2) {$+$};
\node at (2,-2) {$+$};
\node at (3,-2) {$+$};
\node at (4,-2) {$+$};
\node at (5,-2) {$\cdot$};

\node at (0,-3) {$+$};
\node at (1,-3) {$+$};
\node at (2,-3) {$+$};
\node at (3,-3) {$\cdot$};
\node at (4,-3) {$\cdot$};

\node at (0,-4) {$\cdot$};
\node at (1,-4) {$\cdot$};
\node at (2,-4) {$\cdot$};
\node at (3,-4) {$\cdot$};

\node at (0,-5) {$\cdot$};
\node at (1,-5) {$\cdot$};
\node at (2,-5) {$\cdot$};

\node at (0,-6) {$\cdot$};
\node at (1,-6) {$\cdot$};

\node at (0,-7) {$\cdot$};

\end{tikzpicture}
\caption{\label{fig:dbot} Bottom pipe dream $D_\Bot(u)$ for $u=14862357\in S_8\ts$.}

\end{figure}
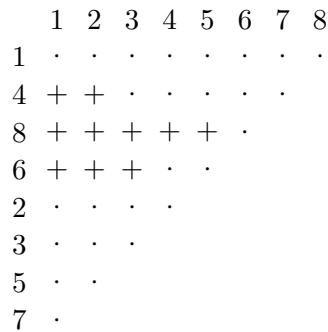
\end{ex}

For a pipe dream $D\in \pd(u)$, denote by $\cL(D)\subseteq \pd(u)$ the set of pipe dreams for $u$ that can be obtained from $D$ by a sequence of ladder moves. We will need the following key combinatorial observation:

\smallskip

\begin{prop}[{\rm Bergeron--Billey \cite[Thm.~3.7(c)]{BB93}}{}]\label{t:schubpd}
    Let \. $u\in S_n$. Then \. $\cL(D_\Bot(u))=\pd(u)$.
\end{prop}

\medskip

\section{From pipe dreams to ladder sequences}\label{s:defs}

In this section, we introduce ladder indices and the corresponding ladder order. Then, we define ladder sequences.  We will show how every pipe dream can be encoded by a certain ladder sequence whose entries count the sizes of patches of crosses in the pipe dream. Finally, we introduce a system of inequalities on the set of ladder sequences which (as we will prove in $\S$\ref{s:bijection}) characterize the ladder sequences coming from pipe dreams.

\subsection{Ladder sequences}\label{ss:ladderseq}

Let $\mu\in \nn_{\ge 1}$. Let
\[\cL_\mu:=\{(i,k_1,k_2,\ldots, k_\ell)\in \nn_{\ge 1}^{\ell+1}\. : \.  \ell\in \nn_{\ge 1}, \; i+k_1+\cdots +k_\ell\le \mu\}.\]
We call $\cL_\mu$ the set of \defn{ladder indices}. To simplify notation, let
\[\bi:=(i,k_1,\ldots, k_\ell), \quad \bi':=(i',k_1',\ldots, k_{\ell'}'),\quad \bi'':=(i'', k_1'',\ldots, k_{\ell''}'').\]
For $\bi\in \cL_\mu$, we also define
\[|\bi|:=i+k_1+\cdots +k_\ell.\]

We define the \defn{ladder order} \ts on \ts $\cL_\mu$ \ts as follows.
We say \. $(i,k_1,\ldots, k_\ell)\prec (i', k_1',\ldots, k_{\ell'}')$ \. if
\begin{itemize}
    \item [(a)] $i>i'$, or
    \item [(b)] $i=i'$, and there exists \. $j\in \{0,\ldots, \min(\ell,\ell')-1\}$ \. such that \. $k_1=k_1', k_2=k_2',\ldots, k_j=k_j'$ and $k_{j+1}<k_{j+1}'$, or
    \item [(c)] $i=i'$, $\ell>\ell'$, and $k_1=k_1', k_2=k_2', \ldots, k_\ell=k_\ell'\ts$.
\end{itemize}
We note that this is a total order on $\cL_\mu$.

\begin{ex}\label{ex:ladderorder}
    For $\mu=4$, the ordering on $\cL_\mu$ is:
    \small
    \[(3,1)\prec (2,1,1)\prec (2,1)\prec (2,2)\prec (1,1,1,1)\prec (1,1,1)\prec (1,1,2)\prec (1,1)\prec (1,2,1)\prec (1,2)\prec (1,3).\]
    \small
\end{ex}

For $\mu\in \nn_{\ge 1}$, let $\CS_\mu$ be the set of $\nn$-valued vectors indexed by $\cL_\mu$. We call $\CS_\mu$ the set of \defn{ladder sequences}.

\begin{ex}\label{ex:laddersequence}
    Let $\mu=4$. Then, the vector
    \[x:=(0,0,2,1,0,0,1,1,0,0,0)\]
    is an element of $\CS_4$ with
    \[x_{(2,1)}=2, \quad x_{(2,2)}=1, \quad x_{(1,1,2)}=1, \quad x_{(1,1)}=1,
    \quad \text{and} \ \  x_\bi=0 \ \ \text{otherwise.} \]
\end{ex}

Given \ts $u\in S_n$ \ts with \. $\mu:=\max\{i\. : \.  c_i(u)>0\}$, the \defn{$u$-weight} \ts
of a ladder sequence \ts $x\in \CS_\mu$ \ts is the vector \ts $\wtx_u(x):=(\cw_1,\cw_2,\ldots)$, where \ts $\cw_j$ \ts is the number
\begin{equation}\label{eq:uwtdef}
    \cw_j\, := \, c_j(u) \. - \. \sum_{\bi\in \cL_\mu, \ts|\bi|\ts =\ts j}x_\bi \. + \. \sum_{\bi\in \cL_\mu,\ts i\ts=\ts j}x_\bi\..
\end{equation}

\smallskip

\subsection{Encoding pipe dreams}\label{ss:pdalgo}
We can now present an algorithm that encodes pipe dreams as ladder sequences.

\smallskip

\begin{lemma}\label{l:inverseladder}
    Let $u\in S_n$ and let $D$ be a pipe dream for $u$. Let $i$ be maximal such that there exists $j$ with $(i,j)\notin D$ and $(i,j+1)\in D$. Let $j$ be maximal with $(i,j)\notin D$ and $(i,j+1)\in D$. Let $k> j$ be minimal such that $(i,k+1)\notin D$. Then, one can perform the inverse ladder moves \. $L_{i,j+1}^{-1}, L_{i,j+2}^{-1},\ldots, L_{i,k}^{-1}$ \ts in this order.
\end{lemma}

\begin{proof}
    First, note that by the choice of $i$, the crosses in all rows below $i$ are left-aligned. Now, let $\ell> i$ be minimal such that there is some $m\in \{j,j+1,\ldots k\}$ with $(\ell,m)\notin D$. Such an $\ell$ exists because $D$ is finite. Let $m\in \{j,j+1,\ldots, k\}$ be minimal with $(\ell,m)\notin D$. Then, if $m>j$, we know that $D$ contains the following the diagram:
    \begin{figure}[H]
    \begin{tikzpicture}[scale=0.5]

\node at (-2,0) {$i$};
\node at (-2,-1) {$i+1$};
\node at (-2,-2) {$\vdots$};
\node at (-2,-3.5) {$\ell-1$};
\node at (-2,-4.5) {$\ell$};

\node at (0,1) {$j$};
\node at (2,1) {$j+1$};
\node at (4,1) {$j+2$};
\node at (6,1) {$\cdots$};
\node at (8,1) {$m-1$};
\node at (10,1) {$m$};
\node at (12,1) {$\cdots$};
\node at (14,1) {$k$};
\node at (16,1) {$k+1$};

\node at (0,0) {$\cdot$};
\node at (2,0) {$+$};
\node at (4,0) {$+$};
\node at (6,0) {$\cdots$};
\node at (8,0) {$+$};
\node at (10,0) {$+$};
\node at (12,0) {$\cdots$};
\node at (14,0) {$+$};
\node at (16,0) {$\cdot$};

\node at (0,-1) {$+$};
\node at (2,-1) {$+$};
\node at (4,-1) {$+$};
\node at (6,-1) {$\cdots$};
\node at (8,-1) {$+$};
\node at (10,-1) {$+$};
\node at (12,-1) {$\cdots$};
\node at (14,-1) {$+$};
\node at (16,-1) {$*$};

\node at (0,-2) {$\vdots$};
\node at (2,-2) {$\vdots$};
\node at (4,-2) {$\vdots$};
\node at (6,-2) {$\vdots$};
\node at (8,-2){$\vdots$};
\node at (10,-2) {$\vdots$};
\node at (12,-2) {$\vdots$};
\node at (14,-2) {$\vdots$};
\node at (16,-2) {$\vdots$};

\node at (0,-3.5) {$+$};
\node at (2,-3.5) {$+$};
\node at (4,-3.5) {$+$};
\node at (6,-3.5) {$\cdots$};
\node at (8,-3.5) {$+$};
\node at (10,-3.5) {$+$};
\node at (12,-3.5) {$\cdots$};
\node at (14,-3.5) {$+$};
\node at (16,-3.5) {$*$};

\node at (0,-4.5) {$+$};
\node at (2,-4.5) {$+$};
\node at (4,-4.5) {$+$};
\node at (6,-4.5) {$\cdots$};
\node at (8,-4.5) {$+$};
\node at (10,-4.5) {$\cdot$};
\node at (12,-4.5) {$\cdots$};
\node at (14,-4.5) {$\cdot$};
\node at (16,-4.5) {$\cdot$};
\end{tikzpicture}
\end{figure}

The $*$'s represent either crosses or dots.
This cannot be a valid pipe dream since the pipe entering row $\ell$
from the left crosses the pipe entering column $j$ from below twice,
once at $(\ell,j)$ and once at $(i,m)$.
Hence, we must have $m=j$ and the diagram instead looks like:
\begin{figure}[H]
    \begin{tikzpicture}[scale=0.5]

\node at (-2,0) {$i$};
\node at (-2,-1) {$i+1$};
\node at (-2,-2) {$\vdots$};
\node at (-2,-3.5) {$\ell-1$};
\node at (-2,-4.5) {$\ell$};

\node at (0,1) {$j$};
\node at (2,1) {$j+1$};
\node at (4,1) {$j+2$};
\node at (6,1) {$\cdots$};
\node at (8,1) {$k$};
\node at (10,1) {$k+1$};

\node at (0,0) {$\cdot$};
\node at (2,0) {$+$};
\node at (4,0) {$+$};
\node at (6,0) {$\cdots$};
\node at (8,0) {$+$};
\node at (10,0) {$\cdot$};

\node at (0,-1) {$+$};
\node at (2,-1) {$+$};
\node at (4,-1) {$+$};
\node at (6,-1) {$\cdots$};
\node at (8,-1) {$+$};
\node at (10,-1) {$*$};

\node at (0,-2) {$\vdots$};
\node at (2,-2) {$\vdots$};
\node at (4,-2) {$\vdots$};
\node at (6,-2) {$\vdots$};
\node at (8,-2) {$\vdots$};
\node at (10,-2) {$\vdots$};

\node at (0,-3.5) {$+$};
\node at (2,-3.5) {$+$};
\node at (4,-3.5) {$+$};
\node at (6,-3.5) {$\cdots$};
\node at (8,-3.5) {$+$};
\node at (10,-3.5) {$*$};

\node at (0,-4.5) {$\cdot$};
\node at (2,-4.5) {$\cdot$};
\node at (4,-4.5) {$\cdot$};
\node at (6,-4.5) {$\cdots$};
\node at (8,-4.5) {$\cdot$};
\node at (10,-4.5) {$\cdot$};
\end{tikzpicture}
\end{figure}
It is clear from definitions that we can perform the inverse ladder move $L_{i,j+1}^{-1}$ and then we may perform $L_{i,j+2}^{-1}$ and so on until $L_{i,k}^{-1}$.
\end{proof}

\smallskip

\begin{lemma}\label{l:pdalgo}
    Let $u\in S_n$, let $\mu:=\max\{i\. : \.  c_i(u)>0\}$.  Then there is an
    injection \ts $\varphi_u:\pd(u)\to \CS_\mu$ \ts that is weight-preserving:
 \ts $\wtx_u(\varphi_u(D))=\wtx(D)$ \ts for all \ts $D\in \pd(u)$.
\end{lemma}

\begin{proof}
    Let $D\in \pd(u)$. The following algorithm will repeatedly apply inverse ladder moves to $D$ until we get the bottom pipe dream $D_b:=D_\Bot(u)$ for $u$.
    \begin{itemize}
        \item [(1)] Choose $i$ maximal such that there exists $j$ with $(i,j)\notin D$ and $(i,j+1)\in D$. If no such $i$ exists, all of the crosses in $D$ are left aligned, $D=D_b$, and we are done.
        \item [(2)] Choose $j$ maximal such that $(i,j)\notin D$ and $(i,j+1)\in D$. Let $k> j$ be minimal such that $(i,k+1)\notin D$.
        \item [(3)] Perform the inverse ladder moves $L_{i,j+1}^{-1}, L_{i,j+2}^{-1},\ldots ,L_{i,k}^{-1}$ in succession to obtain a new pipe dream $D'$ for $u$.
        \item [(4)] Repeat (1) through (3) on $D':=L_{i,k}^{-1}\circ \cdots \circ L_{i,j+2}^{-1}\circ L_{i,j+1}^{-1}(D)$.
    \end{itemize}

    Lemma~\ref{l:inverseladder} ensures that we can perform the necessary inverse ladder moves in step (3). Note that after performing the inverse ladder moves in step (3), if the crosses that we moved down are not left-aligned, they will necessarily be moved down again the next time we perform step (3). In addition, note that the proof of Lemma~\ref{l:inverseladder} shows that the inverse ladder moves $L_{i,j+1}^{-1},\ldots, L_{i,k}^{-1}$ each move a cross down the same number of rows.

    Thus, we can think of this algorithm as picking the bottom-most, right-most patch of non-left-aligned crosses and repeatedly applying inverse ladder moves to each cross in the patch until it is left-aligned and then repeating this process on the next patch.

    Then, we set $\varphi_u(D):=x$ where $x_{(i,k_1,\ldots, k_\ell)}$ is the number of crosses in the (necessarily unique) patch that starts in row $i$ of $D$ and is moved down first $k_1$ rows, then $k_2$ rows and so on until it is left-aligned after performing $\ell$ inverse ladder moves on each cross in the patch. Note that this patch of crosses ends up in row $|\bi|=i+k_1+\cdots +k_\ell$ and is not moved again.

    After performing all of the inverse ladder moves in the above algorithm, we end up with $D_b$. Since $\mu=\max\{i\. : \.  c_i(u)>0\}$, we know that no cross ends up below row $\mu$. This shows that $x_{\bi}$ can only be nonzero if $|\bi|\le \mu$. Hence, $\varphi_u(D)\in \CS_\mu$ and $\varphi_u$ is well-defined.

    The ladder order on $\cL_\mu$ precisely corresponds to the order in which we perform the inverse ladder moves on the patches in the above algorithm. Namely, fix a pipe dream $D\in \pd(u)$ and let $x:=\varphi_u(D)$. Then, for all $\bi,\bi'\in \cL_\mu$ with $x_\bi, x_{\bi'}\neq 0$ we have $\bi\prec \bi'$ if and only if the patch corresponding to $\bi$ is moved before the patch corresponding to $\bi'$.

    Thus, the fact that $\varphi_u$ is injective follows from noting that we can recover $D$ from $\varphi_u(D)$ by starting with $D_b$ and undoing all of the inverse ladder moves (with the corresponding ladder moves) done in the above algorithm. Reading the tuple $x=\varphi_u(D)$ backwards gives the number of crosses on which the ladder moves should be done in each step. In particular, if $x_{\bi}=j$, then we perform $\ell$ ladder moves on each of the right-most $j$ crosses in row $|\bi|$, moving them up first $k_\ell$ rows, then $k_{\ell-1}$ rows, and so on until they end in row $i$. See Example~\ref{ex:pdalgo} below.

    Unraveling the definition of the $u$-weight in the context of the above algorithm, we see that the $j$-th entry of $\wtx_u(\varphi_u(D))$ is $c_j(u)$ minus the number of crosses that move up from row $j$ when doing the ladder moves to obtain $D$ from $D_b$ plus the number of crosses that move up into row $j$ during this process. This is precisely the total number of crosses that end up in row $j$ of $D$ which is indeed the $j$-th entry of $\wtx(D)$ by definition.
\end{proof}

\smallskip

\begin{rem}\label{r:numpatches}
    Let \. $u\in S_n$, \. $\mu:=\max\{i\. : \.  c_i(u)>0\}$, \. $D\in \pd(u)$, and \ts $x:=\varphi_u(D)$. By keeping track of how the crosses move when applying the algorithm in Lemma~\ref{l:pdalgo}, we can observe that for each pair \ts $1\le i <j\le \mu$, there is at most one index \ts $\bi=(i,k_1,\ldots, k_\ell)\in \cL_\mu$ \ts with \ts $x_\bi\neq 0$ \ts and \ts $|\bi|=j$.
    In other words, when performing the algorithm in Lemma~\ref{l:pdalgo} to obtain $D_\Bot(u)$ from $D$, there is at most one patch of crosses that starts in row $i$ and is moved down to row $j$ for each $1\le i<j\le \mu$.
\end{rem}

\smallskip

\begin{ex}\label{ex:pdalgo}
    Continuing Example~\ref{ex:pd}, we have $\mu=4$ so only the first four rows of the pipe dream $D$ will be changed by the algorithm in Lemma~\ref{l:pdalgo}. Omitting the labels and the bottom four rows for simplicity, the sequence of inverse ladder moves is shown in Figure~\ref{fig:pdalgo}.

\begin{figure}[hbt]
    \begin{tikzpicture}[scale=0.47]

\node at (-6, -1) {$D=$};

\node at (-4,0) {$\cdot$};
\node at (-3,0) {$\cdot$};
\node at (-2,0) {$+$};
\node at (-1,0) {$\cdot$};
\node at (0,0) {$+$};
\node at (1,0) {$\cdot$};
\node at (2,0) {$\cdot$};
\node at (3,0) {$\cdot$};

\node at (-4,-1) {$+$};
\node at (-3,-1) {$\cdot$};
\node at (-2,-1) {$+$};
\node at (-1,-1) {$\cdot$};
\node at (0,-1) {$+$};
\node at (1,-1) {$+$};
\node at (2,-1) {$\cdot$};

\node at (-4,-2) {$+$};
\node at (-3,-2) {$+$};
\node at (-2,-2) {$+$};
\node at (-1,-2) {$\cdot$};
\node at (0,-2) {$\cdot$};
\node at (1,-2) {$\cdot$};

\node at (-4,-3) {$+$};
\node at (-3,-3) {$\cdot$};
\node at (-2,-3) {$\cdot$};
\node at (-1,-3) {$\cdot$};
\node at (0,-3) {$\cdot$};

\node at (6,-1) {$\xrightarrow{L_{2,6}^{-1}\. \circ \. L_{2,5}^{-1}}$};

\node at (9,0) {$\cdot$};
\node at (10,0) {$\cdot$};
\node at (11,0) {$+$};
\node at (12,0) {$\cdot$};
\node at (13,0) {$+$};
\node at (14,0) {$\cdot$};
\node at (15,0) {$\cdot$};
\node at (16,0) {$\cdot$};

\node at (9,-1) {$+$};
\node at (10,-1) {$\cdot$};
\node at (11,-1) {$+$};
\node at (12,-1) {$\cdot$};
\node at (13,-1) {$\cdot$};
\node at (14,-1) {$\cdot$};
\node at (15,-1) {$\cdot$};

\node at (9,-2) {$+$};
\node at (10,-2) {$+$};
\node at (11,-2) {$+$};
\node at (12,-2) {$+$};
\node at (13,-2) {$+$};
\node at (14,-2) {$\cdot$};

\node at (9,-3) {$+$};
\node at (10,-3) {$\cdot$};
\node at (11,-3) {$\cdot$};
\node at (12,-3) {$\cdot$};
\node at (13,-3) {$\cdot$};

\node at (18,-1) {$\xrightarrow{L_{2,3}^{-1}}$};

\node at (20,0) {$\cdot$};
\node at (21,0) {$\cdot$};
\node at (22,0) {$+$};
\node at (23,0) {$\cdot$};i
\node at (24,0) {$+$};
\node at (25,0) {$\cdot$};
\node at (26,0) {$\cdot$};
\node at (27,0) {$\cdot$};

\node at (20,-1) {$+$};
\node at (21,-1) {$\cdot$};
\node at (22,-1) {$\cdot$};
\node at (23,-1) {$\cdot$};
\node at (24,-1) {$\cdot$};
\node at (25,-1) {$\cdot$};
\node at (26,-1) {$\cdot$};

\node at (20,-2) {$+$};
\node at (21,-2) {$+$};
\node at (22,-2) {$+$};
\node at (23,-2) {$+$};
\node at (24,-2) {$+$};
\node at (25,-2) {$\cdot$};

\node at (20,-3) {$+$};
\node at (21,-3) {$+$};
\node at (22,-3) {$\cdot$};
\node at (23,-3) {$\cdot$};
\node at (24,-3) {$\cdot$};

\node at (-4,-6) {$\xrightarrow{L_{1,5}^{-1}}$};

\node at (-2,-5) {$\cdot$};
\node at (-1,-5) {$\cdot$};
\node at (0,-5) {$+$};
\node at (1,-5) {$\cdot$};
\node at (2,-5) {$\cdot$};
\node at (3,-5) {$\cdot$};
\node at (4,-5) {$\cdot$};
\node at (5,-5) {$\cdot$};

\node at (-2,-6) {$+$};
\node at (-1,-6) {$\cdot$};
\node at (0,-6) {$\cdot$};
\node at (1,-6) {$+$};
\node at (2,-6) {$\cdot$};
\node at (3,-6) {$\cdot$};
\node at (4,-6) {$\cdot$};

\node at (-2,-7) {$+$};
\node at (-1,-7) {$+$};
\node at (0,-7) {$+$};
\node at (1,-7) {$+$};
\node at (2,-7) {$+$};
\node at (3,-7) {$\cdot$};

\node at (-2,-8) {$+$};
\node at (-1,-8) {$+$};
\node at (0,-8) {$\cdot$};
\node at (1,-8) {$\cdot$};
\node at (2,-8) {$\cdot$};

\node at (7,-6) {$\xrightarrow{L_{2,4}^{-1}}$};

\node at (9,-5) {$\cdot$};
\node at (10,-5) {$\cdot$};
\node at (11,-5) {$+$};
\node at (12,-5) {$\cdot$};
\node at (13,-5) {$\cdot$};
\node at (14,-5) {$\cdot$};
\node at (15,-5) {$\cdot$};
\node at (16,-5) {$\cdot$};

\node at (9,-6) {$+$};
\node at (10,-6) {$\cdot$};
\node at (11,-6) {$\cdot$};
\node at (12,-6) {$\cdot$};
\node at (13,-6) {$\cdot$};
\node at (14,-6) {$\cdot$};
\node at (15,-6) {$\cdot$};

\node at (9,-7) {$+$};
\node at (10,-7) {$+$};
\node at (11,-7) {$+$};
\node at (12,-7) {$+$};
\node at (13,-7) {$+$};
\node at (14,-7) {$\cdot$};

\node at (9,-8) {$+$};
\node at (10,-8) {$+$};
\node at (11,-8) {$+$};
\node at (12,-8) {$\cdot$};
\node at (13,-8) {$\cdot$};

\node at (18,-6) {$\xrightarrow{L_{1,3}^{-1}}$};

\node at (20,-5) {$\cdot$};
\node at (21,-5) {$\cdot$};
\node at (22,-5) {$\cdot$};
\node at (23,-5) {$\cdot$};
\node at (24,-5) {$\cdot$};
\node at (25,-5) {$\cdot$};
\node at (26,-5) {$\cdot$};
\node at (27,-5) {$\cdot$};

\node at (20,-6) {$+$};
\node at (21,-6) {$+$};
\node at (22,-6) {$\cdot$};
\node at (23,-6) {$\cdot$};
\node at (24,-6) {$\cdot$};
\node at (25,-6) {$\cdot$};
\node at (26,-6) {$\cdot$};

\node at (20,-7) {$+$};
\node at (21,-7) {$+$};
\node at (22,-7) {$+$};
\node at (23,-7) {$+$};
\node at (24,-7) {$+$};
\node at (25,-7) {$\cdot$};

\node at (20,-8) {$+$};
\node at (21,-8) {$+$};
\node at (22,-8) {$+$};
\node at (23,-8) {$\cdot$};
\node at (24,-8) {$\cdot$};

\end{tikzpicture}
\caption{\label{fig:pdalgo}Applying the algorithm of Lemma~\ref{l:pdalgo} to obtain $D_\Bot(u)$ from $D$ by a sequence of inverse ladder moves}
\end{figure}

Thus, using the ordering of Example~\ref{ex:ladderorder}, we can see that
\[x:=\varphi_u(D)=(0,0,2,1,0,0,1,1,0,0,0)\in \CS_4,\]
where the nonzero entries correspond to
\[x_{(2,1)}=2, \quad x_{(2,2)}=1, \quad x_{(1,1,2)}=1, \quad x_{(1,1)}=1.\]
We can then recover $D$ by starting with $D_\Bot(u)$ and performing the following ladder moves as determined by reading the nonzero entries of $x$ backwards:
\begin{itemize}
    \item [1.] $x_{(1,1)}=1$ corresponds to performing one ladder move on the right-most cross of row $2$ to move it up $1$ row. This is achieved by doing $L_{2,2}=(L_{1,3}^{-1})^{-1}$.
    \item [2.] $x_{(1,1,2)}=1$ corresponds to performing two ladder moves on the right-most cross of row $4$ to move it up $2$ rows and then $1$ row. This is achieved by doing $L_{4,3}=(L_{2,4}^{-1})^{-1}$ then $L_{2,4}=(L_{1,5}^{-1})^{-1}$.
    \item [3.] $x_{(2,2)}=1$ corresponds to performing one ladder move on the right-most cross of row $4$ to move it up $2$ rows. This is achieved by doing $L_{4,2}=(L_{2,3}^{-1})^{-1}$.
    \item [4.] $x_{(2,1)}=2$ corresponds to performing one ladder move on the right-most $2$ crosses of row $3$ to move them up $1$ row. This is achieved by doing $L_{3,5}=(L_{2,6}^{-1})^{-1}$ then $L_{3,6}=(L_{2,5}^{-1})^{-1}$.
\end{itemize}
As before, we have $\wtx(D)=(2,4,3,1)$ and we can see that
\[\wtx_u(x)=(0+1+1,\ts 2-1+1+2, \ts 5-2, \ts 3-1-1)=(2,4,3,1)=\wtx(D).\]
\end{ex}

\smallskip

\subsection{Compatible ladder sequences}\label{ss:ineqs}

Fix $u\in S_n$, let $\mu:=\max\{i\. : \.  c_i(u)>0\}$ where $c_i(u)$ is the $i$-th entry of $\code(u)$, and fix a subset $\vk\subseteq \cL_\mu$. Recall that we write $\bi,\bi',\bi''\in \cL_\mu$ for $(i,k_1,\ldots, k_\ell), (i', k_1',\ldots, k_{\ell'}'), (i'', k_1'',\ldots, k_{\ell''}'')$ respectively. For each $\bi\in \vk$ and $i< j \le |\bi|$, define
\begin{equation}\label{eq:cijdef}
    c_j^\bi \. : \. \CS_\mu\to \Z\., \qquad c_j^\bi(x)\. := \. c_j(u)-\sum_{\bi'} x_{\bi'}\.,
\end{equation}
where the summation is over
\[\bi'\in \vk \qquad \text{such that} \qquad \bi\prec \bi' \quad \text{and}\quad |\bi'|=j.\]
Define
\[\vk^\bi \. := \. \{\bi'\in \vk\. : \.  \bi\prec \bi' \text{ and } i=i'\}.\]
We define $c_i^\bi$ differently according to whether $\vk^\bi$ is empty or not. If $\vk^\bi=\emp$, then let
\begin{equation}\label{eq:ciidefempty}
    c_i^\bi\. : \. \CS_\mu\to \Z\., \qquad c_i^\bi(x)\. := \. c_i(u)-\sum_{\bi'}x_{\bi'}\.,
\end{equation}
where the summation is over
\[\bi'\in \vk \qquad \text{such that} \qquad \bi\prec \bi' \quad \text{and}\quad |\bi'|=i.\]
Note that this is the same as \eqref{eq:cijdef}. On the other hand, if $\vk^\bi\neq \emp$, let $\bi'$ be the minimal element of $\vk^\bi$ with respect to the ladder order. Then, define
\begin{equation}\label{eq:ciidefnonempty}
    c_i^\bi\.:\.\CS_\mu\to \Z\.,\qquad c_i^\bi(x)\. := \. c_{|\bi'|}^\bi(x)+x_{\bi'}+\ell'\.,
\end{equation}
where $|\bi'|> i$ so $c_{|\bi'|}^\bi$ is defined by \eqref{eq:cijdef}.

We define the set of \defn{$(\vk,u)$-compatible ladder sequences}, denoted \ts $\CS_{\vk, u}$, to be the subset \ts
$\CS_{\vk, u}\subseteq\CS_\mu$ \ts consisting of elements $x\in \CS_\mu$ that satisfy the following conditions:
\begin{itemize}
\item [$\bullet$] For all $\bi\in \cL_\mu\ts$,
    \begin{equation}\label{eq:ineqa}
        \bi \in \vk \iff x_\bi \neq 0.
    \end{equation}
\item [$\bullet$] For all \ts $j\in [\mu]$,
    \begin{equation}\label{eq:ineqb}
        \sum_{\bi\in \vk, \ts |\bi|\ts =\ts j} x_\bi\le c_j(u).
    \end{equation}
\item [$\bullet$] For all \ts $\bi\in \vk$, $\, p\in [\ell]$, $\, q\in [k_p]$,
    \begin{itemize}
        \item [$\circ$] If \ts $q=k_p\ts$, then
        \begin{equation}\label{eq:ineqci}
            c_{i+k_1+\cdots + k_{p-1}}^\bi(x)\, \le \, c_{|\bi|}^\bi(x) -x_\bi +\ell -p.
        \end{equation}
        \item [$\circ$] If \ts $q\neq k_p\ts$, then
        \begin{equation}\label{eq:ineqcii}
            c_{i+k_1+\cdots + k_{p-1}+q}^\bi(x) \, \ge \, c_{|\bi|}^\bi(x) +\ell -p+1.
        \end{equation}
    \end{itemize}
\end{itemize}

\smallskip
Then, define the set of \defn{$u$-compatible ladder sequences}, denoted $\CS_u$ to be
\[
\CS_u\. := \. \bigcup_{\vk \ts \subseteq \ts \cL_\mu} \CS_{\vk,u}\..
\]

\begin{rem}\label{r:ineqform}
    We can see that all of the inequalities describing $\CS_{\vk,u}$ are of the form
    \[
    \sum_{\bi\ts \in\ts \vk} \. a_{\bi}x_{\bi} \, \le \, c_j(u)-d\cdot c_{j'}(u) \. + \. b,
    \]
    for some \. $a_{\bi}\in \{0,1,-1\}, j,j'\in [\mu],d\in \{0,1\}$ \. and \ts $b\in \Z$ \ts that are independent of~$u$. Here we allow $d=0$ in order to get inequalities of the form \eqref{eq:ineqb} that only involve one entry of $\code(u)$. We can get inequalities of the form \eqref{eq:ineqa} that involve no entries of $\code(u)$ by setting $d=1, j=j'$.
\end{rem}
\begin{rem}\label{r:disjoint}
    Note also that for any $\vk\neq  \vk'\subseteq\cL_\mu$, condition \eqref{eq:ineqa} immediately implies that \ts
    $\CS_{\vk, u}\cap \CS_{\vk',u}=\emp$, so $\CS_u$ is a disjoint union of the $\CS_{\vk,u}$'s.
    Hence, we have \. $|\CS_u|=\sum_{\vk \ts \subseteq \ts \cL_\mu} |\CS_{\vk,u}|.$
\end{rem}

\medskip

\section{Proving the bijection}\label{s:bijection}

In this somewhat technical section, we prove that the map defined in Lemma~\ref{l:pdalgo} is a bijection between the set of pipe dreams for a permutation and the set of ladder sequences that satisfy the inequalities described in $\S$\ref{ss:ineqs}. In this way, we reformulate the Schubert coefficient problem as a problem of counting lattice points in a polytope.

\subsection{The setup}\label{ss:bij-setup}

Fix $u\in S_n$ and let $\mu:=\max\{i\. : \.  c_i(u)>0\}$. In Lemma~\ref{l:pdalgo}, we defined a weight-preserving injection $\varphi_u:\pd(u)\to \CS_\mu$. In $\S$\ref{ss:ineqs}, we described inequalities on the set $\CS_\mu$ to define the subset $\CS_u\subseteq\CS_\mu$. Let $\Im(\varphi_u)\subseteq\CS_\mu$ be the image of $\varphi_u$. In this section, we show that
\[\CS_u = \Im(\varphi_u),\]
so $\varphi_u:\pd(u)\to \CS_u$ is a weight-preserving bijection.

\smallskip

\subsection{The image of $\varphi_u$ is $u$-compatible}\label{ss:bij-proof}
To prove that \ts $\varphi_u$ \ts is a bijection, we start by showing the inclusion:

\smallskip

\begin{prop}\label{prop:imsubSu}
    Let \ts $u\in S_n$.  Then \. $\Im(\varphi_u)\subseteq \CS_u\ts$.
\end{prop}

\begin{proof}
    Let $\mu:=\max\{i\. : \.  c_i(u)>0\}$. We perform downward induction on the weight of a pipe dream with respect to the reverse lexicographic order. The base case is the unique maximal element $D_\Bot(u)$ which has weight $\wtx(D_\Bot(u))=\code(u)$. It is clear from the algorithm in Lemma~\ref{l:pdalgo} that $\varphi_u(D_\Bot(u))=\mathbf{0}$ and $\mathbf{0} \in \CS_{\emp, u}\subseteq \CS_u$.

    Now, let $D\in \pd(u)$ and assume that for any $\wt{D}\in \pd(u)$ with $\wtx(\wt{D})>_\text{\rm revlex} \wtx(D)$, we know that $\varphi_u(\wt{D})\in \CS_u$. Let $\wt{D}$ be the pipe dream obtained from $D$ by applying the algorithm in Lemma~\ref{l:pdalgo} until one patch of crosses becomes left-aligned. For ease of notation, define
    \[x:=\varphi_u(D)\qquad \text{and} \qquad \wt{x}:=\varphi_u(\wt{D}).\]

    By definition of $\varphi_u$, we can see that $\wt{x}$ is equal to $x$ with the first (in the ladder order) nonzero entry replaced by zero. Let $\bi\in \cL_\mu$ be the index for which $\wt{x}_{\bi}=0$ and $x_{\bi}\neq 0$.

    Since we have moved crosses down to obtain $\wt{D}$ from $D$, we have $\wtx(\wt{D})>_\text{revlex} \wtx(D)$ so  $\wt{x}\in \CS_u$ by the induction hypothesis. Let $\wt{\vk}\subseteq \cL_\mu$ be the unique indexing set with $\wt{x}\in \CS_{\wt{\vk}, u}$ and let $\vk:=\wt{\vk}\cup \{\bi\}$.

    \medskip
    \nin
    {\bf Claim:} \ts $x\in \CS_{\vk, u}\subseteq  \CS_u$ and thus $\Im(\varphi_u)\subseteq \CS_u$.

    \emph{Proof.} Since $\wt{x}\in \CS_{\wt{\vk}, u}$, \eqref{eq:ineqa} tells us that
    \[\wt{\vk}=\{\bi'\in \cL_\mu\. : \.  \wt{x}_{\bi'}\neq 0\}\]
    and thus, by construction,
    \[\vk=\{\bi'\in \cL_\mu\. : \.  \wt{x}_{\bi'}\neq 0\}\cup \{\bi\}=\{\bi'\in \cL_\mu\. : \.  x_{\bi'}\neq 0\},\]
    so $x$ satisfies \eqref{eq:ineqa} to be in $\CS_{\vk, u}$.

    Since $x=\wt{x}$ except at index $x_\bi$ and since $\wt{x}$ satisfies \eqref{eq:ineqb} for all $j\in [\mu]$, we can see that $x$ satisfies \eqref{eq:ineqb} for all $j\in [\mu]-|\bi|$. For $j=|\bi|$, \eqref{eq:ineqb} says that the number of crosses that leave the $j$-th row of $D_\Bot(u)$ to make $D$ is less than or equal to the $j$-th entry of $\code(u)$. This must be true since $D$ is a valid pipe dream that can be obtained by starting with $\wt{D}$ and moving $x_\bi$ crosses up from row $|\bi|$. Thus, $x$ satisfies \eqref{eq:ineqb} for all $j\in [\mu]$.

    Now, we show that $x$ satisfies \eqref{eq:ineqci} and \eqref{eq:ineqcii} for all $\bi'\in \vk, p\in [\ell'],$ and $q\in [k_p']$. For a given index $\bi'\in \vk$ and all $i'\le j \le |\bi'|$, note that $c_j^{\bi'}(x)$ only involves indices $\bi''\in \vk$ with $\bi''\succ \bi'$. Thus, for a given index $\bi'\in \vk$, \eqref{eq:ineqci} and \eqref{eq:ineqcii} only involve indices $\bi''\in \vk$ with $\bi''\succcurlyeq \bi'$. We know that
    \[\wt{x}_{\bi''}=x_{\bi''}\qquad \text{for all}\quad \bi''\succ \bi.\]

    Hence, the only inequalities of the form \eqref{eq:ineqci} or \eqref{eq:ineqcii} that are different for $x$ and $\wt{x}$ are those corresponding to indices $\bi'\in \vk$ with $\bi'\preccurlyeq \bi$.  By construction, the only such $\bi'$ is $\bi'=\bi$. So since $\wt{x}$ satisfies these inequalities for all $\bi'$, we only need to show that $x$ satisfies \eqref{eq:ineqci} and \eqref{eq:ineqcii} for the index $\bi$. We simplify notation by letting
    \[c_j:=c_j^\bi(x) \qquad \text{for each}\quad  i\le j\le |\bi|.\]

    \medskip
    \nin
    {\bf Claim:} \ts For all $i\le j\le |\bi|,$ the number $c_j$ is the column index of the right-most cross in row $j$ of~$\wt{D}$.
    \begin{proof}[Proof of claim]

    First, let us treat the case of $j>i$. Since $\wt{x}$ is equal to $x$ except with $\wt{x}_{\bi}=0$, \eqref{eq:cijdef} gives
    \[c_j=c_j(u)- \sum_{\bi'} \wt{x}_{\bi'},\]
    where the summation is over
    \[\bi'\in \vk \qquad \text{such that} \qquad \bi\prec \bi' \quad \text{and} \quad |\bi'|=j.\]
    Then, we know that $\wt{x}_{\bi'}=0$ for all $\bi'\notin \wt{\vk}\subseteq \vk$ and for all $\bi'\preccurlyeq \bi$. So
    \[c_j=c_j(u)- \sum_{\bi'} \wt{x}_{\bi'},\]
    where the summation is taken over the larger set
    \[\bi'\in \cL_\mu \qquad \text{such that} \qquad |\bi'|=j.\]
    Finally, since $j>i$, we also have $\wt{x}_{\bi'}=0$ for all $\bi'\in \cL_\mu$ with $i'=j$. So
    \[\sum_{\bi'\ts \in \ts \cL_\mu,\ts i'\ts =\ts j}\wt{x}_{\bi'}=0,\]
    and we can write
    \[c_j \, = \, c_j(u) \. - \. \sum_{\bi'\ts \in\ts  \cL_\mu,\ts |\bi'|\ts=\ts j} \wt{x}_{\bi'} \. + \. \sum_{\bi'\ts\in\ts\cL_\mu, \ts i'\ts=\ts j} \wt{x}_{\bi'} \, = \, \cw_j,
    \]
    where $\cw_j$ is the $j$-th entry of $\wtx_u(x)$ by \eqref{eq:uwtdef}. Then, Lemma~\ref{l:pdalgo} implies that this is the $j$-th entry of $\wtx(D)$ since $x=\varphi_u(D)$. Note that all of the crosses in row $j$ of $\wt{D}$ for $j>i$ are left-aligned so $c_j$ is thus the column index of the right-most cross in row $j$ of $\wt{D}$, proving the claim for $j>i$.

    Now, we deal with the $j=i$ case. First, suppose that $\vk^{\bi}=\emp$. This means that the patch of crosses in $D$ corresponding to $\bi$ is the left-most patch of non-left-aligned crosses in row $i$. By choice of $\bi$, it is also the right-most patch of crosses in row $i$ of $D$.

    Thus, since we left-align this patch of crosses to obtain $\wt{D}$ from $D$, we can see that the crosses in rows $i$ and below of $\wt{D}$ are all left-aligned. In particular, the column index of the right-most cross in row $i$ of $\wt{D}$ is just the number of crosses in row $i$ of $\wt{D}$ which is $c_i$ by a similar argument to the $j>i$ case and by \eqref{eq:ciidefempty}.

    Second, suppose that $\vk^{\bi}\neq \emp$ and let $\bi'$ be the minimal element of $\vk_{\bi}$. Now, by \eqref{eq:ciidefnonempty}, we have
    \[c_i=c_{|\bi'|}+x_{\bi'}+\ell'.\]

    By the $j>i$ case, we know that $c_{|\bi'|}$ is the number of crosses in row $|\bi'|$ of $\wt{D}$. We know that in forming $\wt{D}$ from $D_\Bot(u)$, the last set of ladder moves we do is to move a patch of $x_{\bi'}$ crosses up from row $|\bi'|$ to row $i'=i$ in $\ell'$ steps. It is easy to see that $c_i$ is the column index of the right-most cross in this patch. By choice of $\bi'$, this is the right-most cross in row $i$ of $\wt{D}$, proving the claim.
    \end{proof}

    By construction, we can obtain $D$ from $\wt{D}$ by performing $\ell$ sets of ladder moves on the right-most $x_{\bi}$ crosses that start in row $|\bi|$ of $\wt{D}$. The first set of ladder moves move each cross in the patch up $k_\ell$ rows, the second $k_{\ell-1}$ and so on with the $\ell$-th moving each cross in the patch up $k_1$ rows into row $i$.

    Then, fixing $p\in [\ell]$ and letting $q\in [k_p]$ vary, it can be easily checked that \eqref{eq:ineqci} and \eqref{eq:ineqcii} correspond to the fact that the $(\ell-p+1)$-th set of ladder moves that move this patch up $k_{p}$ rows are all valid ladder moves.

    Since these are all valid ladder moves, $x$ satisfies the inequalities in \eqref{eq:ineqci} and \eqref{eq:ineqcii} corresponding to the index $\bi$ and so $x\in \CS_{\vk, u}$, showing the claim and the finishing the proof of the proposition.
\end{proof}

\smallskip

\subsection{Map $\varphi_u$ is a bijection}\label{ss:bij-proof2}
We now prove the opposite inclusion:

\smallskip

\begin{prop}\label{prop:Susubim}
    Let $u\in S_n$.  Then \. $\CS_u\subseteq \Im(\varphi_u)$.
\end{prop}

\begin{proof}
    Let $\mu:=\max\{i\. : \.  c_i(u)>0\}$ and suppose $x\in \CS_{\vk, u}$ for some $\vk\subseteq \cL_\mu$. We induct on the size of $\vk$. The base case is $\vk=\emp$, for which it is clear that $\CS_{\emp, u}=\{\mathbf{0}\}$. As noted in the proof of Proposition~\ref{prop:imsubSu}, we have \ts
    $\varphi_u(D_\Bot(u))=\mathbf{0}$, which implies \ts $\mathbf{0} \in \Im(\varphi_u)$ \ts and \ts $\CS_{\emp, u}\subseteq \Im(\varphi_u)$.

    Now, let $\vk\subseteq \cL_\mu$ and assume that for all \ts $\wt{\vk}\subseteq \cL_\mu$ \ts with $|\wt{\vk}|<|\vk|$, we have \ts $\CS_{\wt{\vk}, u}\subseteq \Im(\varphi_u)$. Fix \ts $x\in \CS_{\vk, u}$. We want to show that \ts $x\in \Im(\varphi_u)$, i.e., \ts $x=\varphi_u(D)$ \ts for some pipe dream~$D$. Let \ts $\wt{x}\in \CS_\mu$ \ts be such that $\wt{x}$ is equal to $x$ with the first nonzero entry replaced by zero. Let \ts $\bi\in \cL_\mu$ \ts be the index for which \ts $\wt{x}_{\bi}=0$ \ts and \ts $x_{\bi}\neq 0$.

    \medskip
    \nin
    {\bf Claim:} \ts Let $\wt{\vk}:=\vk - \{\bi\}$. Then, $\wt{x}\in \CS_{\wt{\vk},u}$.
    \begin{proof}[Proof of Claim]
     By construction, $\wt{x}$ satisfies \eqref{eq:ineqa} for $\wt{\vk}$ and it is clear that $\wt{x}$ also satisfies \eqref{eq:ineqb}. As in the proof of Proposition~\ref{prop:imsubSu}, for a given index $\bi'$, \eqref{eq:ineqci} and \eqref{eq:ineqcii} only involve indices $\bi''\in \wt{\vk}$ with $\bi''\succcurlyeq \bi'$. But $\wt{x}_{\bi''}=x_{\bi''}$ for all $\bi''\succ \bi$ so $\wt{x}$ satisfies \eqref{eq:ineqci} and \eqref{eq:ineqcii} for any $\bi'\succ \bi$ since $x$ does. Then, $\wt{\vk}=\{\bi'\in \vk\. : \.  \bi'\succ \bi\}$ by choice of $\bi$ minimal in $\vk$ so indeed $\wt{x}$ satisfies all of the required conditions and is in $\CS_{\wt{\vk},u}$.  This proves the claim.
    \end{proof}

    Thus, as $|\wt{\vk}|<|\vk|$, we have $\wt{x}\in \Im(\varphi_u)$ by the induction hypothesis. Let $\wt{D}\in \pd(u)$ be such that $\varphi_u(\wt{D})=\wt{x}.$

    \medskip
    \nin
    {\bf Claim:} \ts We can obtain a pipe dream $D$ from $\wt{D}$ by performing $\ell$ sets of ladder moves on the right-most $x_{\bi}$ crosses that start in row $|\bi|$ of $\wt{D}$. The first set of ladders moves will move each cross in the patch up $k_\ell$ rows, the second $k_{\ell-1}$ rows and so on with the $\ell$-th moving each cross in the patch up $k_1$ rows into row~$i$.

    \begin{proof}[Proof of Claim]
    As in the proof of Proposition~\ref{prop:imsubSu}, for each $i\le j\le |\bi|$, we have that $c_j^\bi(x)$ is the column index of the right-most cross in row $j$ of $\wt{D}$. Then, the fact that $x$ satisfies \eqref{eq:ineqb} for $j=|\bi|$, and satisfies \eqref{eq:ineqci} and \eqref{eq:ineqcii} for the index $\bi$ ensure that all of these ladder moves are valid so there is indeed such a pipe dream $D$ for~$u$. This proves the claim.
    \end{proof}

    Moreover, this construction and the fact that \ts $\varphi_u(\wt{D})=\wt{x}$, implies that \ts $\varphi_u(D)=x$.  Indeed, when performing the algorithm of Lemma~\ref{l:pdalgo}, we first move the patch corresponding to $\bi$ down until it is left-aligned and record how many crosses were in that patch (which is $x_\bi$ by construction), as the $\bi$-th entry of  $\varphi_u(D)$. After this, we obtain the pipe dream~$\wt{D}$, so continuing the algorithm gives \ts $\varphi_u(D)=\varphi_u(\wt{D})$ \ts for all larger indices. Hence, $x\in \Im(\varphi_u)$ so $\CS_{\vk, u}\subseteq \Im(\varphi_u)$, completing the proof of Proposition~\ref{prop:Susubim}.
\end{proof}

\smallskip

\begin{thm}\label{t:bij}
    Let \ts $u\in S_n$. Then \. $\varphi_u:\pd(u)\to \CS_u$ \. is a weight-preserving bijection.
\end{thm}

\begin{proof}
    This follows immediately from Lemma~\ref{l:pdalgo}, Proposition~\ref{prop:imsubSu} and  Proposition~\ref{prop:Susubim}.
\end{proof}

\smallskip

\begin{rem}\label{r:stretchedunion}
    For $u\in S_n$ with $\mu:=\max\{i\. : \.  c_i(u)>0\}$, note that for all $N \ge 1$, we have \.
    $\max\{i\. : \.  c_i(N*u)>0\}=\mu$ \. since \. $c_i(N*u)=Nc_i(u)$ for all $i$. Thus,
    \begin{equation}\label{eq:stretchedunion}
        \Im(\varphi_{N*u})\, = \, \CS_{N*u} \, = \, \bigsqcup_{\vk \ts \subseteq \ts \cL_\mu}\.\CS_{\vk, N*u}\ts.
    \end{equation}
    In other words, all possible pipe dreams for \ts $N*u$ \ts are described by the same sets of
    inequalities, independent of \ts $N\ge 1.$  The parameter~$N$ appears only in the constraints given
    by the code.
\end{rem}

\subsection{An example}\label{ss:bij-ex}
    Continuing Example~\ref{ex:pdalgo}, let
    \[\vk:=\{\bi \in \cL_4\. : \.  x_\bi \neq 0\}=\{(2,1),(2,2),(1,1,2),(1,1)\}.\]

    Let us look at the inequalities defining $\CS_{\vk, u}$. For $j\in [4]$, \eqref{eq:ineqb} becomes
    \begin{align*}
        j=1: \qquad &0\le 0,\\
        j=2: \qquad &1=x_{(1,1)}\le 2,\\
        j=3: \qquad &2=x_{(2,1)}\le 5,\\
        j=4: \qquad &1+1=x_{(2,2)} + x_{(1,1,2)}\le 3.
    \end{align*}
    Now, we specifically consider the index $\bi:=(1,1,2)$, so
    \[i=1, \quad \ell=2, \quad k_1=1,\quad k_2=2, \quad |\bi|=1+1+2=4, \quad \text{and} \quad x_\bi=1.\]
    Then, by \eqref{eq:cijdef}, we have:
    \begin{align*}
        c_2^{\bi}(x) &= 2-1+0=1,\\
        c_3^{\bi}(x) &= 5-0+0=5,\\
        c_4^{\bi}(x) &= 3-0+0=3.
    \end{align*}
    Also, note that $\vk^{\bi}=\{(1,1)\}$ so $\bi':=(1,1)$ is the minimal element of $\vk^{\bi}$ with respect to the ladder order. Since $|\bi'|=1+1=2$, we have (by \eqref{eq:ciidefnonempty}):
    \[c_1^{\bi}(x) = c_{|\bi'|}^{\bi}(x)+x_{(1,1)}+1=1+1+1=3.\]

    For $p\in [2], q\in [k_p]$, using $|\bi|=4,\ts x_\bi=1$, and the above computations, the inequalities in \eqref{eq:ineqci} and \eqref{eq:ineqcii} translate to:
    \begin{align*}
        p=1, q=1: \qquad &3=c_1^{\bi}\le c_{|\bi|}^{\bi}-x_{\bi}+2-1=3,\\
        p=2, q=1: \qquad &5=c_{1+1+1}^{\bi}\ge c_{|\bi|}^{\bi}+2-2+1=4,\\
        p=2, q=2: \qquad &1=c_{1+1}^{\bi}\le c_{|\bi|}^{\bi}-x_{\bi}+2-2=2.
    \end{align*}
    Similar calculations for the other indices in~$\vk$ \ts show that \ts $x\in \CS_{\vk,u}\ts$, as expected.

\medskip

\section{Proof of Theorems~\ref{t:main} and~\ref{t:cs}}\label{s:proof}

In this section, we combine the weight-preserving bijection of Section~\ref{s:bijection} with Theorem~\ref{t:kostasum} and a result from Ehrhart theory to prove Theorem~\ref{t:main}. In particular, we write a Schubert coefficient as an alternating sum of the number of triples of pipe dreams satisfying certain weight conditions and then use the eventual quasi-polynomiality of the number of lattice points in dilated polytopes.

\subsection{Geometric lemma}\label{ss:cited}
As we mentioned in the introduction, we need the following result in Ehrhart theory:

\smallskip

\begin{lemma}[{\rm Chen--Lil--Sam~\cite[Lemma~3.3]{CLS12}}{}]\label{l:quasipoly}
     Let \ts $P(N)\subseteq\rr^n$ \ts be a polytope defined by inequalities of the form \.
     $A\mathbf{x}\le N\mathbf{c}+\mathbf{b}$, where $A$ is an \ts $k\times n$ \ts
     integer matrix, and \ts $\mathbf{c},\mathbf{b}\in \Z^k$.
     Then the function \. $F_{A,\mathbf{c},\mathbf{b}}(N) :=\big|P(N)\cap \Z^n\big|$ \. is eventually a quasi-polynomial.
\end{lemma}

\smallskip

\begin{rem}\label{r:eventualbound}
Following the original proof of \cite[Lem.~3.3]{CLS12}, one can easily obtain
the following effective extension of Lemma~\ref{l:quasipoly}:
    \[F_{A,\mathbf{c},\mathbf{b}}(N) \ \text{ is a quasi-polynomial in} \  N, \ \  \text{ for } \
    N \. > \. |\mathbf{b}|_\infty \cdot 2^{n-1},
    \]
    where \. $|\mathbf{b}|_\infty\ts :=\ts \max\{|b_i|\. : \.  1\le i \le k\}$.
\end{rem}

\begin{rem}\label{r:period-bound}
    By Ehrhart's theorem, the period $m$ of the Ehrhart quasi-polynomial of a polytope~$P$
    divides the lowest common multiple of the denominators of the vertices of~$P$, see e.g.\ \cite[Thm~3.23]{BR07}.
    In this case, the denominator of every vertex divides
    the determinant of some square submatrix of~$A$.  In other words, we have the following bound:
     \begin{equation}\label{eq:period-bound}
     m \, \le \, \lcm \big\{|\det(Y)| \. : \. Y \subseteq A\} \, \le \, \big(\max_{Y \subseteq A} |\det(Y)|\ts \big)! \ts,
     \end{equation}
     where the \ts $\lcm$ \ts and the maximum are taken over all square submatrices of~$A$.
\end{rem}

\smallskip

\subsection{Triples of pipe dreams}\label{ss:pdtrips}
Now that we have described the possible pipe dreams for $N*u$ in terms of systems of inequalities,
we need the following  result to finish the proof of Theorem~\ref{t:main}.

\begin{lemma}\label{l:pdtrips}
Fix $u,v,w\in S_n$. Also, let $\sigma\in S_{Nn}$. For each $N\ge 1$, define the set
    \begin{align*}
        A_\sigma^N:=\{(P_1,P_2,P_3)\. : \.  P_1\in \pd(N*u)&, \, P_2\in \pd(N*v),\,P_3\in \pd(\wh w_{\circ}(N*w)),\\ &\wtx(P_1)+\wtx(P_2)+\wtx(P_3)=\sigma\. \wh\rho{}\},
    \end{align*}
    where \. $\wh w_{\circ}:=(Nn,Nn-1,\ldots, 2,1)\in S_{Nn}$ \. and \.
    $\wh \rho:=(Nn-1,\ldots,1,0)\in \nn^{Nn}$ \. on which  \ts $\sigma\in S_{Nn}$ \ts
    acts by permuting the elements as follows:
    \[\sigma\. \wh\rho \, := \, \big(Nn-\sigma(1),\ldots, Nn-\sigma(Nn)\big).
    \]
Then \. $f_\sigma:\nn_{\ge 1}\to \nn$ \. defined by \. $f_\sigma(N):=|A_\sigma^N|$ \. is eventually a quasi-polynomial in~$N$.
\end{lemma}

\begin{proof}
    For simplicity, we adopt the following notation:
    \[\Nu:=N*u, \quad \Nv:= N*v, \quad \Nw:=\wh w_{\circ}(N*w).\]
    Define
    \[\mu:=\max\{i\. : \.  c_i(u)>0 \. \text{ or } \. c_i(v)>0 \. \text{ or } \. c_i(w)>0\}\le n-1.\]

    Then, we have \. $c_i(\Nu)=N\cdot c_i(u)=0$ \. for \ts $i> \mu$.  Similarly, we have \. $c_i(\Nv)=0$ \. for \ts $i> \mu$. We also have $c_i(\Nw)=N(n-c_i(w))-i=Nn-i$ for $i> \mu$. Recall that the bottom pipe dream $D_\Bot(\Nu)$ has weight $\wtx(D_\Bot(\Nu))=\code(\Nu)$ and any other pipe dream for $\Nu$ has smaller weight in the reverse lexicographic order. Hence, for any $i>\mu$ and any $P_1\in \pd(\Nu)$, the $i$-th entry of $\wtx(P_1)$ is zero. We have the same result for any $P_2\in \pd(\Nv)$.

    Hence, the only way to have \. $P_1\in \pd(\Nu),\, P_2\in \pd(\Nv), \, P_3\in \pd(\Nw)$ \. with
    \[\wtx(P_1) \. + \. \wtx(P_2)\. + \. \wtx(P_3) \, = \, \sigma \ts \wh\rho{}\]
    is if the $i$-th entry of $\wtx(P_3)$ is $Nn-i$ for $i>\mu$ and $\sigma(i)=i$ for $i> \mu$. Thus, if $f_\sigma$ is not identically zero, we may treat $\sigma\in S_{\mu}$ as just permuting the first $\mu$ entries of $\wh\rho{}$. We may also assume that the pipe dreams in any triple $(P_1,P_2,P_3)\in A_\sigma^N$ can be obtained from $(D_\Bot(\Nu), D_\Bot(\Nv), D_\Bot(\Nw))$ respectively by performing a sequence of ladder only on crosses starting in the first $\mu$ rows. Namely, they are given by ladder sequences in $\Phi_\mu$.

    Using the weight-preserving bijections $\varphi_{\Nu},\ts  \varphi_{\Nv},$ and $\varphi_{\Nw}$ of Theorem~\ref{t:bij}, and by Remarks~\ref{r:disjoint} and~\ref{r:stretchedunion}, we see that
    \begin{align}
        \big|A_\sigma^N\big|&=\big|\big\{x\in \CS_{\Nu},\. y\in \CS_{\Nv},\. z\in \CS_{\Nw}\. : \.
        \wtx_{\Nu}(x)+\wtx_{\Nv}(y)+\wtx_{\Nw}(z)=\sigma \ts\wh\rho{}\big\}\big|\\
        &= \sum_{\vk_1,\vk_2,\vk_3\ts \subseteq \ts \cL_\mu} \big|\big\{x\in \CS_{\vk_1, \Nu},\. y\in \CS_{\vk_2, \Nv},\. z\in \CS_{\vk_3, \Nw}\. : \.   \wtx_{\Nu}(x)+\wtx_{\Nv}(y)+\wtx_{\Nw}(z)=\sigma \ts\wh\rho{}\big\}\big|\label{eq:pdtripsum}
    \end{align}

    Then, for fixed $\vk_1,\vk_2,\vk_3\subseteq \cL_\mu$, since $c_i(\Nu)=Nc_i(u)$, Remark~\ref{r:ineqform} tells us that $x\in \CS_{\vk_1,\Nu}$ if and only if $x$ satisfies a system of inequalities of the form
    \[\sum_{\bi\ts \in\ts \vk_1} a_{\bi} x_{\bi}\, \le \, N(c_j(u)-d\cdot c_{j'}(u))+b,
    \]
    for some
    \[a_{\bi}\in \{0,1,-1\},\quad j,j'\in [\mu],\quad d\in \{0,1\},\quad  b\in \Z\]
    that are independent of $N$. We have a similar system of inequalities corresponding to $y\in \CS_{\vk_2,\Nv}$. For $z\in \CS_{\vk_3,\Nw}$, we note that $c_j(\Nw)=N(n-c_j(w))-j$ so $z$ satisfies a system of inequalities of the form
    \[\sum_{\bi\ts \in\ts \vk_3} a_{\bi}z_{\bi}\le N(n(d-1)+c_j(w)-d\cdot c_{j'}(w))+b,\]
    with $a_\bi,j,j',d,b$ as above.

    Using \eqref{eq:uwtdef}, we see that $\wtx_{\Nu}(x)+\wtx_{\Nv}(y)+\wtx_{\Nw}(z)=\sigma \wh\rho{}$ corresponds to $x,y,z$ satisfying the following equations for $1\le j \le \mu$:
    \begin{align*}
        N(c_j(u)+&c_j(v)+n-c_j(w))-j-\sum_{\bi\ts\in \ts\vk_1,\ts |\bi|\ts=\ts j}x_{\bi}\.-\sum_{\bi\ts\in \ts\vk_2,\ts |\bi|\ts=\ts j}y_{\bi}\.-\sum_{\bi\ts\in \ts\vk_3,\ts |\bi|\ts=\ts j}z_{\bi}\\
        &+\sum_{\bi\ts\in\ts \vk_1, \ts i\ts=\ts j}x_\bi\.+\sum_{\bi\ts\in\ts \vk_2, \ts i\ts=\ts j}y_\bi\.+\sum_{\bi\ts\in\ts \vk_3, \ts i\ts =\ts j}z_\bi=Nn-\sigma(j).
    \end{align*}

    We can write each of these equations as two inequalities where, as above, the relevant coefficients are independent of $N$. Putting together the systems of inequalities coming from
    \[x\in \CS_{\vk_1, \Nu}\. , \quad y\in \CS_{\vk_2,\Nv}\., \quad z\in \CS_{\vk_3,\Nw}\.,\]
    and this final system of $2\mu$ inequalities, we can see that
    $$\Xi(N) \, := \, \big\{x\in \CS_{\vk_1, \Nu}, \. y\in \CS_{\vk_2, \Nv}, \. z\in \CS_{\vk_3, \Nw}\, : \,   \wtx_{\Nu}(x)+\wtx_{\Nv}(y)+\wtx_{\Nw}(z)\. = \.\sigma \ts\wh\rho{}\big\}
    $$
    is the set of integer points of a polytope $P(N)$ defined by inequalities of the form \. $A\mathbf{x}\le N\mathbf{c}+\mathbf{b}$ as in Lemma~\ref{l:quasipoly}.  Thus, its size is eventually a quasi-polynomial in~$N$.
    \end{proof}

\smallskip

\begin{proof}[Proof of Theorem~\ref{t:main}]
Fix $u,v,w\in S_n$ and let $f_{u,v,w}(N):=c_{N*u,N*v}^{N*w}$ for $N\ge 1$. By \eqref{eq:post-stan}, we have
\[
c_{N*u,N*v}^{N*w} \, = \, \sum_{\sigma\ts \in \ts S_{Nn}}\sign(\sigma)\sum_{(a,b,c)\ts\in\ts \Omega(\sigma)} K_{N*u,a}\. K_{N*v,b}\. K_{\wh w_{\circ}(N*w),c} \, = \, \sum_{\sigma \ts\in\ts S_{Nn}}\. \sign(\sigma)\. f_\sigma(N).
\]
The second equality follows from Theorem~\ref{t:schub} and the fact that for a permutation $u$ and a composition $a=(a_1,a_2,\ldots)$, we know that $K_{u,a}$ is the number of pipe dreams with $a_i$ crosses in the $i$-th row. Each $f_\sigma(N)$ is eventually a quasi-polynomial in $N$ by Lemma~\ref{l:pdtrips}.   As noted in the  proof of Lemma~\ref{l:pdtrips}, we only need to sum over \ts $\sigma\in S_{\mu}\ts$ with $\mu\le n-1$.  Therefore, the function \. $f_{u,v,w}(N)=c_{N*u,N*v}^{N*w}$ \. is also eventually a quasi-polynomial in~$N$, as desired.
\end{proof}

\smallskip

\subsection{Effective version}\label{ss:proof-quant}
The bounds in the proofs of Theorem~\ref{t:main} can be made completely explicit,
giving the following effective version of the theorem:

\begin{thm}%
\label{t:quant}
Let \ts $u,v,w\in S_n$ have bounded total range $\psi$ defined
in \eqref{eq:boundedrange}. Then, the function \ts $f_{u,v,w}(N) := c_{N*u,N*v}^{N*w}$ \ts
is quasi-polynomial in~$N$ for all \ts $N>\psi\cdot 2^{3\psi^3}$.
Additionally, for the period $m$ we have
$$
m \, \le \, \big((3 \psi)^{3 \psi^2}\big)!\ts,
$$
and the degree of all polynomials \ts $f_{u,v,w}(N)$ \ts where $N$~{\rm mod}~$m$
is fixed, is at most $4\psi^2$.
\end{thm}

\begin{proof}
Let \. $\tau:=\min\{\Exc(u),\Exc(v),\Exc(w)\}$ \. and note that the first \ts $(\tau-1)$ \ts
elements in \ts $u,v,w$ \ts are fixed points.  The same holds for \ts $N \ast u$,
\ts $N\ast v$ \ts and \ts $N\ast w$.  Removing these fixed points leaves
Schubert coefficients unchanged, so it suffices to consider only the case when
\ts $\tau=1$.  In this case, by the bounded range hypothesis \eqref{eq:boundedrange},
we have
\begin{equation}\label{eq:mupsi}
    \mu \. := \. \max\{\Exc(u),\Exc(v),\Exc(w)\} \. \le \. \psi.
\end{equation}

    Now, by Remark~\ref{r:numpatches}, if \. $\CS_{\vk_1,\Nu}\neq \emp$, then \. $|\vk_1|\le \binom{\mu}{2}$, and similarly for $\vk_2$ and~$\vk_3$. Thus, we are only concerned with sets \ts $\CS_{\vk,\Nu}$ that have at most \. $\binom{\mu}{2}\cdot \mu$ \.
    inequalities of the form \eqref{eq:ineqci} or \eqref{eq:ineqcii}, and similarly for \ts $\Nv$ \ts and \ts $\Nw$.

    Adding this to the at most $\mu$ inequalities of the form \eqref{eq:ineqb} for $\Nu, \Nv$, and $\Nw$ and the at most $2\mu$ inequalities coming from the weight condition, we can see that there are at most $3\mu^3$ inequalities describing the containment in \. $\Xi(N)$.

    Moreover, in each of these inequalities, we can observe from the definitions, that the term independent of $N$ has absolute value at most $\mu$ so by Remark~\ref{r:eventualbound}, we have that $f_\sigma(N)$ is a quasi-polynomial for
    \[N \. > \. \mu \cdot 2^{3\mu^3}.
    \]
    Finally, summing over all possible \ts $\vk_1,\vk_2,\vk_3\subseteq \cL_\mu$ \ts shows that \ts
    $|A_\sigma^N|$ \ts is a quasi-polynomial in $N$ for $N>\mu \cdot 2^{3\mu^3}$, implying the desired bound by \eqref{eq:mupsi}.

\smallskip

     For the period, we use the bound in Remark~\ref{r:period-bound}. In our case,
     the denominator of any vertex divides the determinant of some square submatrix
     of~$A$. We know that $A$ has at most $3\mu^2$ columns, at most $3\mu^3$ rows,
     and all entries in \ts $\{0,\pm 1\}$, so by Hadamard's inequality (see e.g.\ \cite[$\S$2.11]{BB61}),
    the maximum determinant of a submatrix of $A$ is at most $(3\mu^2)^{3\mu^2/2}=(3\mu)^{3\mu^2}$. Thus, the desired bound
    follows from~\eqref{eq:period-bound} and \eqref{eq:mupsi}.

\smallskip

For the degree, we use the following counting argument.
    By Remark~\ref{r:numpatches}, there are at most \ts
    $\mu-i+1$ \ts patches of crosses in row~$i$ of a pipe dream $D$ for $\Nu$. For a given row, choosing where these at most $\mu$ patches go is equivalent to picking at most $\mu$ starting indices and $\mu$ ending indices, all between $1$ and $Nn$. So there are at most
    \[\binom{Nn+2\mu-1}{2\mu-1}= O(N^{2\mu})\]
    such choices for each row between $1$ and $\mu$. Thus, we have
    \[|\pd(\Nu)|= O((N^{2\mu})^\mu)=O(N^{2\mu^2}).\]

     Note that \ts $|\pd(\Nu)|=\Sch_{\Nu}(1,\ldots, 1)$ is
     the principal specialization of the Schubert polynomial.  We similarly have
     $$|\pd(\Nv)| \, = \, \Sch_{\Nv}(1,\ldots, 1)\, = \, O\big(N^{2\mu^2}\big).
     $$
    We conclude:
    \[c_{N*u,N*v}^{N*w}\, \le \,  [x^{\code(N*w)}]\big(\Sch_{\Nu}(x)\cdot \Sch_{\Nv}(x)\big) \, \le \,
    \Sch_{\Nu}(1,\ldots, 1)\cdot \Sch_{\Nv}(1,\ldots, 1) \, = \, O\big(N^{4\mu^2}\big),
    \]
which implies the desired degree bound by \eqref{eq:mupsi}.
\end{proof}

\smallskip

\subsection{Proof of Theorem~\ref{t:cs}}
Recall Barvinok's algorithm, see e.g.\ \cite{Bar08,BP99},
which computes the number \ts $|Q\cap \zz^d|$ \ts
of integer points of a convex polytope \ts $Q\ssu\rr^d$ \ts
with input size~$L$ in time \ts $L^{O(d\log d)}$.
See \cite[Thm~2.3]{PP17} for this explicit bound.

First, as in the proof above, we consider only triples of permutations \ts $u,v,w\in S_n$ \ts with
$$\min\{\Exc(u),\Exc(v),\Exc(w)\} \. = \. 1,
$$
since removing the leading
fixed points leaves Schubert coefficients unchanged.  Thus, in the
notation of the proof above, we always have \ts $\mu\le \psi$.

Note that by Remark~\ref{r:numpatches}, the sum in \eqref{eq:pdtripsum} may be taken over $\vk_1,\vk_2,\vk_3 \subseteq\cL_\psi$ that each have at most one element $\bi\in \vk_k$ (for $k\in \{1,2,3\}$) with $|\bi|=j$ for each pair $1\le i<j\le \psi$. In particular, we can then see that the Schubert coefficient \ts $c^w_{u,v}$ \ts is equal to
the alternating sum of at most
\[T \ := \ \psi! \. \prod_{1\ts\le\ts i\ts<\ts j\ts\le\ts \psi} \big(p(j-i)+1\big)^3
\, \le \, \psi! \, p(\psi)^{3\psi^2} \, = \,  e^{O(\psi^{2.5})}
\]
terms where \ts $p(k)=\exp O(\sqrt{k})$ \ts is the number of integer partitions of~$k$.

Each of these terms is the number of integer points in a $d$-dimensional space, $d=O(\psi^2)$, defined by \ts $O(\psi^3)$ \ts inequalities.
Recall that the matrix entries in $A$ are in \ts $\{0,\pm 1\}$, while the code entries are at most~$n$.
This gives the input size \ts $L=O(\psi^3\log n)$.

Thus, Barvinok's algorithm in this case gives
the following upper bound on the time complexity:
$$
T \cdot L^{O(d\log d)} \, = \, e^{O(\psi^{2.5})} \cdot (\psi\log n)^{O(\psi^2\log \psi)} \, =
\,  e^{O(\psi^{2.5})} (\log n)^{O(\psi^2\log \psi)}\.,
$$
as desired. \qed

\medskip

\section{Examples and counterexamples}\label{s:ex}

In this section, we give proofs of Proposition~\ref{p:negex} and Theorem~\ref{t:offset}. The first proof also serves to give an example of how the encoding of Lemma~\ref{l:pdalgo} actually proceeds for a relatively simple, yet non-trivial example.

\subsection{Proof of Proposition~\ref{p:negex}}\label{ss:ex-negex}
First, we set some notation.

    Fix $N\ge 1$ and let $\sigma\in S_3$. Let $\wh w_{\circ}:=(Nn,Nn-1,\ldots, 2,1)\in S_{Nn}$ and let $\Nu:=N*u, \ts \Nv:=N*w, \ts \Nw:=\wh w_{\circ} (N*w)$. Also, let $\wh\rho{}=(Nn-1,\ldots, 1,0)\in \nn^{Nn}$. As in Lemma~\ref{l:pdtrips}, let
    \[A_\sigma^N:=\{(x,y,z)\. : \.  x\in \CS_{\Nu},\.  y\in \CS_{\Nv},\.  z\in \CS_{\Nw}, \;\wtx_{\Nu}(x)+\wtx_{\Nv}(y)+\wtx_{\Nw}(z)=\sigma \wh \rho{}\}.\]
    Let $x,y,z\in \CS_3$. By inspecting the pipe dreams for $\Nu, \Nv$, and $\Nw$, we can see that $(x,y,z)\in A_\sigma^N$ if and only if the following inequalities are satisfied
    \begin{align*}
        y_{(1,2)}+z_{(2,1)}+z_{(1,1,1)}&\le N-3+\sigma(3),\\
        y_{(1,2)}+y_{(1,1)}+z_{(1,1,1)}&\le N+1-\sigma(1),\\
        y_{(1,2)}+y_{(1,1)}+z_{(2,1)}+2z_{(1,1,1)}&\le 2N-\sigma(1),\\
        y_{(1,2)}+y_{(1,1)}+z_{(1,1,1)}&\ge 2-\sigma(1),
    \end{align*}
    in addition to the following equalities
    \begin{align*}
        x_{(1,1)}=x_{(1,2)}=x_{(1,1,1)}=y_{(2,1)}&=y_{(1,1,1)}=z_{(1,2)}=0,\\
        y_{(1,1)}+y_{(1,2)}+z_{(1,1,1)}+z_{(1,1)} &= N+1-\sigma(1),\\
        x_{(2,1)}+y_{(1,2)}+y_{(2,1)}+z_{(2,1)}+z_{(1,1,1)}&= N-3+\sigma(3).
    \end{align*}

    Here, we have already made the substitutions coming from the weight condition and then removed many redundant inequalities using the fact that all of the variables are nonnegative. In other words, $|A_\sigma^N|$ is equal to the number of nonnegative integral solutions to
    \begin{equation}\label{eq:matrixineq}
        \begin{pmatrix}
            1 & 0 & 1 & 1\\
            1 & 1 & 0 & 1\\
            1 & 1 & 1 & 2\\
            -1 & -1 & 0 & -1
        \end{pmatrix}\, \begin{pmatrix}
            y_{(1,2)}\\y_{(1,1)}\\z_{(2,1)}\\z_{(1,1,1)}
        \end{pmatrix}\ \le \  N \begin{pmatrix}
            1\\1\\2\\0
        \end{pmatrix} \, + \, \begin{pmatrix}
            \sigma(3)-3\\1-\sigma(1)\\-\sigma(1)\\\sigma(1)-2
        \end{pmatrix}.
    \end{equation}

    Then, by a case-by-case inspection of \eqref{eq:matrixineq} and some basic counting arguments (whose details we omit), we get the following results, all of which hold for all $N\ge 1$:

    \begin{align*}
        |A_{123}^N|\, &= \, \frac{(N+1)(N+2)^2(N+3)}{12}\. - \. 2(N+1)\,,\\
        |A_{132}^N| \, &= |A_{213}^N| \, = \, \frac{N(N+1)(N+2)(N+3)}{12} \. - \. N\,,\\
        |A_{231}^N|\, &= |A_{312}^N|\, = \, \frac{(N-1)N(N+1)(N+2)}{12}\,,\\
        |A_{321}^N|\, &= \, \frac{(N-1)N^2(N+1)}{12}\,.
    \end{align*}

    Putting this all together, we get
    \begin{multline*}
        c_{N*u,N*v}^{N*w}\,= \, f_{u,v,w}(N)\, = \, \frac{(N+1)(N+2)^2(N+3)}{12}-2(N+1)\\-2\left(\frac{N(N+1)(N+2)(N+3)}{12}-N\right)+ 2\left(\frac{(N-1)N(N+1)(N+2)}{12}\right)\\-\left(\frac{(N-1)N^2(N+1)}{12}\right) \,= \, N-1,
    \end{multline*}
    which completes the proof.
\qed

\begin{rem}
    Since $2$ is an entry in the matrix in \eqref{eq:matrixineq},
    it is not \emph{totally unimodular} (cf.~\cite[$\S$14]{Bar08}),
    so we are only guaranteed quasi-polynomiality
    with period $2$ in this case. However, the values of \ts $|A_\sigma^N|$ \ts
    are still polynomial in $N$ due to the specific forms of the inequalities
    involved. We expect this \emph{quasi-period collapse}, to happen in all cases
    as it does for the stretched Littlewood--Richardson coefficients,
    see~$\S$\ref{ss:finrem-conj}.  For more on the quasi-period collapse
    in this context, see \cite{HM08} and references therein.
\end{rem}

\smallskip

\subsection{Proof of Theorem~\ref{t:offset}}\label{ss:ex-offset}
We will actually show that Theorem~\ref{t:offset} holds for $n\ge 4$ if we take $u,v,w\in S_n$ such that
\[\code(u)=\code(v)=(1,0,\ldots, 0,1,0),\qquad \code(w)=(3,0,\ldots, 0,1,0).
\]
Note that this formulation in terms of codes agrees with that given in the statement of Theorem~\ref{t:offset} for $n\ge 7$.

We will need the following definition. Define the \defn{reverse dominance order} on tuples of the same length $(a_1,\ldots, a_k),(b_1,\ldots, b_k)\in \nn^k$ by
\[(a_1,\ldots, a_k)\le_\revdom (b_1,\ldots, b_k)\iff \sum_{i\ts=\ts \ell}^k a_i\le \sum_{i\ts=\ts \ell}^k b_i \ \text{ for all } \ 1\le \ell\le k.\]
Note that for any pipe dream $D$ for $u\in S_n$, since we can obtain $D$ from $D_\Bot(u)$ by a sequence of ladder moves, we have
\begin{equation}\label{eq:revdomweight}
    \wtx(D) \. \le_\revdom \. \wtx(D_\Bot(u)) \. = \. \code(u).
\end{equation}
Here, we pad the weights and code with $0$'s so that they are of length $n$.

Now we are ready to prove Theorem~\ref{t:offset}. We split the proof of into two parts.

\medskip
    \nin
{\bf Claim:}
    \begin{equation}\label{eq:offset2}
        c_{N*u,N*v}^{N*w} \. = \. 0 \quad \text{for all}\quad N > 2n-6.
    \end{equation}
\begin{proof}[Proof of Claim]
    Fix \ts $N>2n-6$, and let
    \[\bal \. := \. \code(N*w) \. = \. (3N, 0,\ldots, 0,N,0).
    \]
    Then, we know that
    \begin{equation}\label{eq:offsetschmult}
        c_{N*u,N*v}^{N*w}\, \le \, \big[x^\bal\big] \big(\Sch_{N*u}(x)\cdot \Sch_{N*v}(x)\big).
    \end{equation}
    By Theorem~\ref{t:schub} and Proposition~\ref{t:schubpd}, we have
    \[
    \Sch_{N*u}(x) \, = \, \sum_{D\ts\in\ts \cL(D_\Bot(N*u))} \. x^{\wtx(D)}.
    \]

    When performing ladder moves on $D_\Bot(N*u)$,
    we can move at most $(n-3)$ \. crosses up from the \ts $(n-1)$-\text{st} row
    to the first row. Thus, for any $a\in \nn^n$ with $a_1>N+n-3$, the coefficient of $x^a$ in $\Sch_{N*u}(x)$ is zero. We have a similar result for $\Sch_{N*v}(x)$. Putting these together shows that
    \[[x^\bal]\big(\Sch_{N*u}(x)\cdot \Sch_{N*v}(x)\big) \, = \, 0\qquad \text{since} \qquad \bal_1\. = \. 3N=2N+N \. > \. 2N+2n-6.
    \]
    Thus, \eqref{eq:offsetschmult} implies that $c_{N*u,N*v}^{N*w}=0$, completing the proof of \eqref{eq:offset2}.
\end{proof}

\medskip
    \nin
{\bf Claim:}
\begin{equation}\label{eq:offset1}
    c_{N*u,N*v}^{N*w}\. = \. 1 \quad \text{for} \quad N= 2n-6.
\end{equation}
    \begin{proof}[Proof of Claim]
        For $N=2n-6$, we observe that there is exactly one pipe dream $D$ for $N*u=N*v$ which has \ts $N+n-3=\frac{3N}{2}$ \. crosses in row $1$ and $N-(n-3)=\frac{N}{2}$ \. crosses in row $n-1$. Namely, we move exactly half of the crosses in row \ts $(n-1)$ \ts of \ts $D_\Bot(N*u)$ up into row~$1$.

        Moreover, it is easy to see that there are no pipe dreams for $N*u$ with crosses only in rows $1$ and $(n-1)$, that have more than $\frac{3N}{2}$ crosses in row $1$. Hence, we have
        \[
        K_{N*u, \ts \bal/2}=1\qquad \text{and} \qquad [x^\bal]\big(\Sch_{N*u}(x)\cdot \Sch_{N*v}(x)\big) \, = \, 1,
        \]
        where
        \[\bal \. := \. \code(N*w) \. = \. (3N,0,\ldots, 0, N, 0).\]

        Thus, by equating the coefficient of $x^\bal$ on both sides of
        \[\Sch_{N*u}(x)\cdot \Sch_{N*v}(x)=\sum_{w'\in S_\infty} c_{N*u,N*v}^{w'}\Sch_{w'}(x),\]
        we see that
        \[1=\sum_{w'\in S_\infty} c_{N*u,N*v}^{w'}K_{w',\bal},\]
        so there is exactly one $w'\in S_\infty$ satisfying
        \begin{equation}\label{eq:w'req}
            c_{N*u,N*v}^{w'}>0 \qquad \text{and} \qquad K_{w',\bal}>0.
        \end{equation}

        In particular, this $w'$ must have $c_{N*u,N*v}^{w'}=1$ and $K_{w',\bal}=1.$ We claim that $w'=N*w$, completing the proof of \eqref{eq:offset1}. Suppose that $w'\in S_\infty$ satisfies the conditions of \eqref{eq:w'req}.

        As in \eqref{eq:offsetschmult}, the fact that \ts $c_{N*u,N*v}^{w'}>0$ \ts
        implies that the coefficient of $\code(w')$ in $\Sch_{N*u}(x)\cdot \Sch_{N*v}(x)$ is nonzero. In particular, we can find pipe dreams $D_u\in \pd(N*u)$ and $D_v\in \pd(N*v)$ with $\wtx(D_u)+\wtx(D_v)=\code(w')$. By \eqref{eq:revdomweight}, we know that $\wtx(D_u)\le_\revdom \code(N*u)$, and similarly \ts $\wtx(D_v)\le_\revdom\code(N*v)$. Combining these relations gives
        \begin{equation}\label{eq:w'revdomtwo}
            \code(w')=\wtx(D_u)+\wtx(D_v)\le_\revdom \code(N*u)+\code(N*v)=(2N,0,\ldots, 0,2N,0).
        \end{equation}

        By Theorem~\ref{t:schub}, the fact that $K_{w',\bal}>0$ implies that there is a pipe dream $D$ for $w'$ with $\wtx(D)=\bal$. Then, we can see by Proposition~\ref{t:schubpd} and the definition of ladder moves that the only value for $\code(w')$ satisfying \eqref{eq:w'revdomtwo} and such that there is a pipe dream $D\in \pd(w')$ with $\wtx(D)=\bal$ is $\code(w')=\bal$ and $D=D_\Bot(w')$. Hence, we have $w'=N*w$.  This completes the proof of \eqref{eq:offset1} and finishes the proof of the theorem.
    \end{proof}

\medskip

\section{Final remarks and open problems} \label{s:finrem}

\subsection{}\label{ss:finrem-conj}
In \cite[Conj.~6.28]{Kir04}, Kirillov also speculated whether the generating function \ts
$P_{u,v,w}(t)$ \ts defined in \eqref{eq:gf}
can only have a pole at \ts $t=1$.  This is equivalent to \ts $f_{u,v,w}(N)$ \ts
being eventually polynomial in~$N$.  For comparison, the eventual quasi-polynomiality
in Theorem~\ref{t:main} implies the weaker statement that all poles of \ts $P_{u,v,w}(t)$ \ts are roots of unity.  We restate Kirillov's speculation as a conjecture:

\smallskip

\begin{conj}\label{conj:evpoly}
    Let \. $u,v,w\in S_n\ts$. Then \. $f_{u,v,w}(N)$ \. is eventually polynomial in~$N$.
\end{conj}

\smallskip

Recall the exact starting point \ts $M = 2n-5$ \ts given in
Theorem~\ref{t:offset}. Our exhaustive computer experiments for \ts $n\le 7$
\ts suggest that this bound might be sharp, i.e.\ \ts $f_{u,v,w}(N)$ \ts is
always polynomial for \ts $N \ge 2n-5$.  It would be interesting to bridge the
gap with the exponential upper bound in Theorem~\ref{t:quant}.

\subsection{}\label{ss:finrem-Kostka}
Let us mention that the argument in our proof of Theorem~\ref{t:main}
immediately gives eventual quasi-polynomiality of the Schubert--Kostka numbers:

\begin{prop}\label{p:Kostka}
For all \ts $u\in S_n$ \ts and \ts $a=(a_1,a_2,\ldots)$, s.t.\ $|a|=\ell(u)$,
the generating function
$$Q_{u,a}(t) \, := \, \sum_{N=1}^\infty \. K_{N*u,Na} \, t^N
$$
is rational.  Moreover, the function \ts $g_{u,a}(N) := K_{N*u,Na}$ \ts
is eventually quasi-polynomial.
\end{prop}

Recall that polynomiality of the stretched Kostka numbers is an easier
problem than the corresponding problem for the LR coefficients.
See \cite{AK23} for an elementary proof via reduction to contingency tables,
and generalizations to flagged and cylindric Kostka numbers.
Thus asking if eventual polynomiality holds for the stretched Schubert--Kostka
coefficients is a natural first step towards Conjecture~\ref{conj:evpoly}.

In a different direction, let us mention that the Schubert--Kostka numbers do 
not satisfy the saturation property, cf.\ a discussion in~$\S$\ref{ss:intro-rel}.
This underscores the difference with the saturation for the (usual) Kostka numbers, 
which has several proofs and generalizations \cite{Ale15,AK23}. 

\begin{prop}\label{ex:Kostka}
    Let \. $u=2143\in S_4$ \. and \. $a=(2)$. Then:
    \[K_{u,a} \. = \. 1\qquad \text{and} \qquad K_{N*u,Na} \. = \. 0 \quad \text{for all} \quad N\ge 2.\]
\end{prop}

We omit the easy proof using pipe dreams.

\subsection{}\label{ss:finrem-unim}
It is worth comparing the eventual quasi-polynomiality in Theorem~\ref{t:main}
with polynomiality of the \emph{bounded inversions sum} \ts studied in
\cite[Thm~1.2]{PR-puzzle}.
In both cases, Schubert coefficients with a one parameter constraint are considered.
Both proofs reduce the problem to counting integer points in convex polytopes.
The polynomiality in the latter case is a property of \emph{order polytopes}
which are known to have integer vertices and unimodular triangulations,
cf.\ \cite{FH23,HPPS21}.

It is a major challenge to give a similar geometric
proof of the polynomiality of stretched LR coefficients discussed in the introduction,
cf.\ \cite{Ale19,DM06,Stu95}.  It is thus very unlikely that this approach can be used
to prove Conjecture~\ref{conj:evpoly}.  Let us also mention a more
conceptual approach to the bounded inversions sum by Stelzer~\cite{Ste25} 
(see also~\cite{Ste26}), which
 proved that eventually polynomiality holds in all classical types.

\subsection{}\label{ss:finrem-Kron}
There is a parallel study of \defn{Kronecker coefficients} \ts of the symmetric
group, generalizing LR coefficients in a different direction:
$$
g(\la,\mu,\nu) \, := \, \<\chi^\la\chi^\mu,\chi^\nu\> \, = \, \frac{1}{n!} \.
\sum_{\si \in S_n} \. \chi^\la(\si) \. \chi^\mu(\si)\. \chi^\nu(\si)\.,
$$
where \ts $\chi^\la$ \ts is an irreducible character of the symmetric
group~$S_n$, \ts $\la\vdash n$.  By definition, \ts $g(\la,\mu,\nu)\in \nn$,
but no known combinatorial interpretation is known \cite{Pak-OPAC,Sta00}.

Motivated by computational complexity applications, Mulmuley proved in
\cite{Mul09} that the \emph{stretched Kronecker coefficients} \ts
$g(N\la,N\mu,N\nu)$ \ts and the \emph{stretched plethysm coefficients}
 are quasi-polynomial.  Later work
\cite{CDW12,PP17} gave  Kronecker coefficients as an alternating
sum of the numbers of integer points in polytopes, and used
Barvinok's algorithm to compute Kronecker coefficients in a
style similar to Theorem~\ref{t:cs}.

In the same paper \cite{Mul09}, Mulmuley asked various positivity
and saturation type questions, some of which were later
disproved in \cite{BOR09,PP20}.  Mulmuley also speculated
about the saturation and quasi-polynomiality for stretched
Schubert coefficients \cite[$\S$3.7]{Mul09}, but this subject
was never investigated until \cite{PR26} and this paper.

Let us also mention \cite{SS16} which proved Stembridge's stability
conjecture on stability of \emph{shifted stretched Kronecker coefficients} \.
$g(\al+N\la,\be+N\mu,\ga+N\nu)$, also motivated by integer point in
polytopes considerations \cite{Ste14}.  It would be interesting to see
if there are analogues of such stability results for Schubert coefficients.

\subsection{}\label{ss:finrem-puzzle}
Note that permutations in Proposition~\ref{p:negex} and Theorem~\ref{t:offset}
have at most three and two descents, respectively, and so do their
stretched versions.  Thus, one could in principle try to prove both results with
a careful analysis of $2$-step and $3$-step puzzles \cite{BKPT16,KZ17}.
We avoided this approach in order to use these proofs both as an example
of our linear inequalities technology, and as an illustration of its power.
On the other hand, it would be interesting to see how the stretching affects
the $2$- and $3$-step puzzle tilings.  Can one confirm the eventual polynomiality
(Conjecture~\ref{conj:evpoly}) at least in these cases?

\medskip

\subsection*{Acknowledgements}
We thank Per Alexandersson, Jes\'us De~Loera, Alexey Garber, Christian Ikenmeyer,
Colleen Robichaux, Karthik Sellakumaran Latha, Ada Stelzer and Sylvester Zhang
for interesting discussions and helpful comments.
The first author was partially supported by the NSF grant CCF-2302173.
The second author was partially supported by the NSF grant DMS-2136090.

\end{document}